\documentclass[10pt]{amsart}
\usepackage[utf8]{inputenc}
\usepackage[margin=1.0in]{geometry}
\usepackage{amssymb,amsmath,amsthm}
\usepackage{amsfonts}
\usepackage{amsmath}
\usepackage{xcolor}
\usepackage{amsthm}
\usepackage{pdflscape}
\usepackage{pgfplots}
\usepackage{mathrsfs}
\usepackage{mathbbol}
\usepackage{multirow}
\usepackage[numbers]{natbib}
\usepackage{enumerate}
\usepackage{enumitem}
\usepackage{hyperref}
\usepackage{tikz-cd}
\usepackage{etoolbox}
\usepackage{setspace}
\hypersetup{
pdftitle={On the quantum differentiation of smooth real-valued functions},
pdfsubject={},
pdfauthor={Kolosov Petro},
pdfkeywords={}
}
\usepackage[symbol]{footmisc}
\usepackage{colortbl}
\usepackage{xcolor}
\definecolor{shadegray}{RGB}{128, 128, 128}
\definecolor{lightgray}{RGB}{200,200,200}
\usepackage{amsaddr}
\usepackage{mathtools} 
\newenvironment{brsm}{
  \bigl[ \begin{smallmatrix} }{%
  \end{smallmatrix} \bigr]}

\theoremstyle{definition}
\newtheorem{thm}{Theorem}[section]
\newtheorem{definition}[thm]{Definition}
\newtheorem{example}[thm]{Example}

\newtheorem{cor}[thm]{Corollary}

\newtheorem{lem}[thm]{Lemma}
\newtheorem{quest}[thm]{Question}

\newtheorem{mylem}{Lemma}[subsection]

\numberwithin{equation}{section}

\date{}
\title{Invariants of Unimodular Quadratic Polynomial Poisson Algebras of Dimension 3}
\author{CHENGYUAN MA}
\address{Department of Mathematics, University of Washington}
\email{c9ma@uw.edu}

\let\oldtocsection=\tocsection
\let\oldtocsubsection=\tocsubsection
\let\oldtocsubsubsection=\tocsubsubsection 
\renewcommand{\tocsection}[2]{\hspace{0em}\oldtocsection{#1}{#2}}
\renewcommand{\tocsubsection}[2]{\hspace{1em}\oldtocsubsection{#1}{#2}}
\renewcommand{\tocsubsubsection}[2]{\hspace{2em}\oldtocsubsubsection{#1}{#2}}
\pgfplotsset{compat=1.18}

\begin{document}

\begin{abstract}
Let $P = \Bbbk[x_1, x_2, x_3]$ be a unimodular quadratic Poisson algebra and let $G$ be a finite subgroup of the graded Poisson automorphism group of $P$. In this paper, we prove a variant of the Shephard-Todd-Chevalley Theorem for the Poisson algebra $P$ and variants of both the Shephard-Todd-Chevalley Theorem and the Watanabe Theorem for its Poisson enveloping algebra $U(P)$ under the induced action of $G$.
\end{abstract}

\maketitle

\pagestyle{plain}

\setcounter{section}{-1}

\setstretch{1.2}

\addtocontents{toc}{\protect\setcounter{tocdepth}{0}}
\section{Introduction}
\addtocontents{toc}{\protect\setcounter{tocdepth}{1}}

Throughout $\Bbbk$ is an algebraically closed field of characteristic 0. Let $A = \Bbbk[x_1, \cdots, x_n]$ and let $G$ be a finite subgroup of the graded automorphism group of $A$. The invariant subalgebra of $A$ under the action of $G$ is
\[
A^G := \left\{a \in A: \phi(a) = a \textnormal{ for all } \phi \in G \right\}.
\]
It is natural to ask: what properties, in particular, what homological properties does the invariant subalgebra $A^G$ satisfy? Two of the earliest answers are encapsulated in the Shephard-Todd-Chevalley Theorem and the Watanabe Theorem, articulated as follows:
\begin{thm}
(Shephard-Todd-Chevalley Theorem, \cite{ST}, \cite{C}) Let $A = \Bbbk[x_1, \cdots, x_n]$ and let $G$ be a finite subgroup of the graded automorphism group of $A$. Then the invariant subalgebra $A^G$ is regular (or equivalently, $A^G \cong A$ as $\Bbbk$-algebras) if and only if $G$ is generated by (pseudo-)reflections. 
\end{thm}
\begin{thm}
(Watanabe Theorem, \cite{W}) Let $A = \Bbbk[x_1, \cdots, x_n]$ and $G$ be a finite subgroup of the graded automorphism group of $A$ containing no (pseudo-)reflections.Then the invariant subalgebra $A^G$ is Gorenstein if and only if $\det (\phi\big\vert_{A_1})$ = 1 for all $\phi \in G$.
\end{thm}

In the following decades, the Shephard-Todd-Chevalley Theorem and the Watanabe Theorem have evolved into pivotal motivations for invariant theory, particularly in non-commutative settings: if $A$ is an Artin-Schelter regular algebra and $G$ is a finite subgroup of the graded automorphism group of $A$, under what conditions on $G$ is the invariant subalgebra $A^G$ Artin-Schelter regular or Artin-Schelter Gorenstein? Artin-Schelter regularity, initially introduced in \cite{AS}, emerges as a non-commutative adaption of the polynomial rings because of its fulfillment of a spectrum of properties inherent to polynomial rings. In a more formal manner:

\begin{definition}
A finitely generated $\Bbbk$-algebra $A$ is called \textit{Artin-Schelter regular} if 
\begin{enumerate}[label = (\arabic*)]
    \item $A$ is connected $\mathbb{N}$-graded: $A$ admits a $\Bbbk$-vector space decomposition $A = \displaystyle{\bigoplus_{n \in \mathbb{N}}A_n}$ such that $A_0 = \Bbbk$ and $A_{i}A_{j} \subseteq A_{i+j}$ for all $i, j \in \mathbb{N}$.

    \item $A$ has finite Gelfand-Kirillov dimension: $\dim_{\Bbbk}A_n$ has polynomial growth. 

    \item $A$ has finite global dimension $d$.

    \item $A$ is Gorenstein: $\text{Ext}_{A}^{i}(\Bbbk, A) \cong \begin{cases}
        0 & i \neq d\\
        \Bbbk(l) & i = d
    \end{cases}$, for some $l \in \mathbb{Z}$. 
\end{enumerate}
\end{definition}

Returning to our discussion of the non-commutative Shephard-Todd-Chevalley question and the non-commutative Watanabe question. There are some established answers, including, but not limited to, universal enveloping algebra of semisimple Lie algebras and Weyl algebras in \cite{AP}, non-PI Sklyanin algebras of global dimension $\geq$ 3 in \cite{KKZ}, skew polynomial rings and quantum matrix algebras in \cite{KKZ2}, down-up algebras in \cite{KKZ3}. Recently, Gaddis, Veerapen, and Wang have proposed an investigation into these questions within the realm of Poisson algebras. Broadly speaking, Poisson algebras are a family of commutative algebras endowed with a non-commutative bracket. In a more rigorous language:
\begin{definition}
A \textit{Poisson algebra} is a commutative $\Bbbk$-algebra $P$ together with a bracket: 
\[
\{-,-\}: P^{\otimes 2} \to P
\]
such that
\begin{enumerate}[label = (\arabic*)]
    \item $(P, \{-,-\})$ is a Lie algebra over $\Bbbk$, namely $\{-,-\}$ satisfies bilinearity, alternativity, anti-commutativity, and the Jacobi identity. 

    \item $\{-,-\}$ satisfies Leibniz rule: $\{a, bc\} = \{a,b\}c + b\{a,c\}$ for all $a, b, c \in P$. 
\end{enumerate}
\end{definition}

Poisson algebras originally emerged in classical mechanics and subsequently assumed a significant role in mathematical physics. In recent decades, Poisson algebras have also garnered attention in pure mathematics. This heightened interest is partially attributed to their approximity with Artin-Schelter regular algebras. One instance of such appromixity is Poisson enveloping algebras.

\begin{definition}
Let $P$ be a Poisson algebra. A (The) \textit{Poisson enveloping algebra} of $P$ is a triple $(U, \alpha, \beta)$:
\begin{itemize}
    \item $U$ is a $\Bbbk$-algebra, 

    \item $\alpha: (P, \cdot) \to U$ is an algebra homomorphism,

    \item $\beta: (P, \{-,-\}) \to U_L$ is a Lie algebra homomorphism,
\end{itemize}
subjecting to the following conditions:
\begin{enumerate}[label=(\arabic*)]
    \item $\alpha(\{a,b\}) = \beta(a)\alpha(b) - \alpha(b)\beta(a)$ for all $a, b \in P$.

    \item $\beta(ab) = \alpha(a)\beta(b) + \alpha(b)\beta(a)$ for all $a, b \in P$.

    \item If $(U', \alpha', \beta')$ is another triple satisfying (1) and (2), then there exists a unique algebra homomorphism $h: U \to U'$ making the following diagram commutative:
    \[
    \begin{tikzcd}
    U \arrow{rrdd}{h}\\
    \\
    P \arrow{rr}[swap]{\alpha', \beta'} \arrow{uu}{\alpha,\beta} && U'
    \end{tikzcd}
    \]
\end{enumerate}
\end{definition}

If $P = \Bbbk[x_1, \cdots, x_n]$ is a quadratic Poisson algebra, then its Poisson enveloping algebra $U(P)$ satisifies a range of preferred properties, including being Artin-Schelter regular. This is one connection between Poisson algebras and Artin-Schelter regular algebras. Another connection is found in the context of semiclassical limits and deformation quantizations. However, this topic will not be discussed in this paper but will be explored in a forthcoming paper.
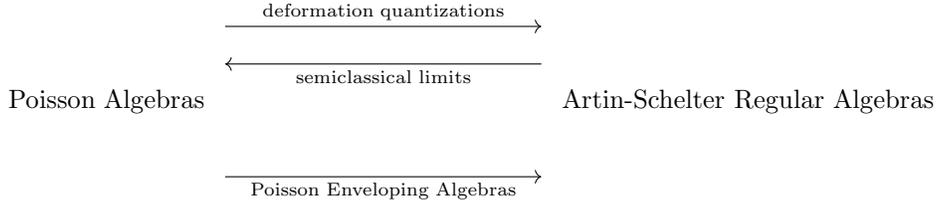
\[
\begin{tikzcd}
\text{Poisson Algebras }
\arrow[rrrrr, shift={(0,1)}, "\text{deformation quantizations}"] 
\arrow[rrrrr, shift={(0,-1)}, "\text{Poisson Enveloping Algebras}", swap]
& & & & & \text{ Artin-Schelter Regular Algebras}
\arrow[lllll, shift={(0,0.5)}, "\text{semiclassical limits}"]
\end{tikzcd}
\]

Given that Poisson algebras exhibit a pronounced association with Artin-Schelter regular algebras through semiclassical limit, deformation quantizations and Poisson enveloping algebras, there is a strong interrelation between investigations pertaining to the Shephard-Todd-Chevalley question and the Watanabe question in the context of Poisson algebras and investigations of these questions in the context of Artin-Schelter regular algebras. In their study \cite{GVW}, Gaddis, Veerapen, and Wang provided a partial answer to the Shephard-Todd-Chevalley question by investigating multiple Poisson structures arising from the semiclassical limits of specific families of Artin-Schelter regular algebras. Additionally, they offered valuable insights on the Watanabe question for Poisson enveloping algebras under induced actions. Building upon their research, we shall further investigate these questions with regard to a broader range of Poisson structures, with a primary emphasis on quadratic Poisson structures on the polynomial ring of three variables $\Bbbk[x_1, x_2, x_3]$.

\smallskip

In Section 2, we establish a series of lemmas essential for the derivation of the main theorems presented in sections 3 and 4 of this paper.

\smallskip

In section 3, we prove one of the main theorem of this paper, a graded rigidity theorem answering the Shephard-Todd-Chevalley question for unimodular quadratic Poisson structures on $\Bbbk[x_1, x_2, x_3]$.
\begin{thm}
(Theorem 3.1) Let $P = \Bbbk[x_1, x_2, x_3]$ be a unimodular quadratic Poisson algebra and let $G \subseteq \text{PAut}_{\text{gr}}(P)$ be a finite subgroup. Then the invariant subalgebra $P^G$ is isomorphic to $P$ as Poisson algebras if and only if $P$ is trivial. 
\end{thm}
For Poisson algebras lacking Poisson reflections, Theorem 3.1 is an immediate consequence of the classical Shephard-Todd-Chevalley Theorem. For the remaining unimodular quadratic Poisson structures on $\Bbbk[x_1, x_2, x_3]$, we provide a case-by-case proof in subsection 3.1 - subsection 3.9. This result offers additional instances of graded rigid Poisson algebras beyond those examined in \cite{GVW}. As it transpires, most invariant subalgebras of unimodular quadratic Poisson structures on $\Bbbk[x_1, x_2, x_3]$ under Poisson reflection groups fail to preserve unimodularity. There exists two invariant subalgebras that retain unimodularity; nonetheless, they are not isomorphic to the original Poisson algebra. Up to this point, all examples of Poisson algebras considered are graded rigid.

\smallskip

In section 4, we shift our attention to the study of the invariants of Poisson enveloping algebras. Let $P = \Bbbk[x_1,\cdots,x_n]$ be a quadratic Poisson algebra and let $U(P)$ be its Poisson enveloping algebra. A graded Poisson automorphism $\phi$ of $P$ induces a graded algebra automorphism $\widetilde{\phi}$ of $U(P)$:
\begin{align*}
    \widetilde{\phi}(x_i) = \phi(x_i), \quad \widetilde{\phi}(y_i) = \sum_{j=1}^{n}\frac{\partial g(x_i)}{\partial x_j}y_j,
\end{align*}
for all $1 \leq i \leq n$. In addition, a subgroup $G$ of the graded Poisson automorphism group of $P$ induces a corresponding subgroup $\widetilde{G}$ of the graded automorphism group of $U(P)$ by setting $\widetilde{G} = \{\widetilde{\phi}: \phi \in G\}$. In Lemma 2.3 and Lemma 2.4, we explain the rationale for the consideration of such induced actions. Subsequently, in section 4, we will prove the remaining two main theorems of this paper under these induction actions, an answer to the Shephard-Todd- Chevalley question and a formula for computing the homological determinant of an induced automorphism for Poisson enveloping algebras of quadratic Poisson structures on $\Bbbk[x_1, \cdots, x_n]$.

\begin{thm}
(Theorem 4.1) Let $P = \Bbbk[x_1, \cdots, x_n]$ be a quadratic Poisson algebra and $U(P)$ be its Poisson enveloping algebra. Let $G$ be a finite nontrivial subgroup of the graded Poisson automorphism group of $P$ and let $\widetilde{G}$ be the corresponding finite nontrivial subgroup of the graded automorphism group of $U(P)$. The invariant subalgebra $U(P)^{\widetilde{G}}$ is Artin-Schelter regular if and only if $G$ is trivial.
\end{thm}

\begin{thm}
(Theorem 4.4) Let $P = \Bbbk[x_1, \cdots, x_n]$ be a quadratic Poisson algebra and let $U(P)$ be its Poisson enveloping algebra. Let $\phi$ be a finite-order graded Poisson automorphism of $P$ and let $\widetilde{\phi}$ be the induced graded automorphism of $U(P)$. Then
\[
\text{hdet} \widetilde{\phi} = (\det \phi\big\vert_{P_1})^2.
\]
\end{thm}

\smallskip

The foundation for both Theorem 4.1 and Theorem 4.4 lies in the observations made in Lemma 2.5: if a graded Poisson automorphism $\phi$ of a quadratic Poisson structure on $\Bbbk[x_1, \cdots, x_n]$ has eigenvalues $\lambda_1, \cdots, \lambda_m$, with multiplicity $c_1, \cdots, c_m$, then the induced automorphism $\widetilde{\phi}$ of the Poisson enveloping algebra has $\lambda_1, \cdots, \lambda_m$ with multiplicity $2c_1, \cdots, 2c_m$. It is noteworthy that this observation provides a generalization of \cite[Theorem 5.6]{GVW}. For Theorem 4.1, we compare the eigenvalues of $\widetilde{\phi}$ with the eigenvalues of quasi-reflections of $U(P)$ described in \cite[Theorem 3.1]{KKZ} and conclude that $\widetilde{\phi}$ cannot be a quasi-reflection. For Theorem 4.4, we derive a formula for computing the trace series of $\widetilde{\phi}$ and relate the trace series to the homological determinant of $\widetilde{\phi}$ as in \cite[Lemma 2.6]{JZ}.

\smallskip

Before delving further into the results in this paper, the author would like to thank his advisor, James Zhang, for his invaluable advice and guidance throughout this project and Xingting Wang for his reviews and remarks on this paper.

\

\section{Fundamentals of Poisson Invariant Theory}

Let $P = \Bbbk[x_1,x_2,x_3]$ be a Poisson algebra under the standard grading. The Poisson algebra $P$ is called \textit{quadratic} if $\{P_1, P_1\} \subseteq P_2$. Dufour and Haraki have classified all quadratic Poisson structures on $P$ into 13+1 classes \cite[Theorem 2]{DH}. In this paper, we are primarily interested in the +1 class: a class consisting of Poisson structures of the following form:
\[
\{x_1,x_2\} = \frac{\partial \Omega}{\partial x_3}, \quad
\{x_2,x_3\} = \frac{\partial \Omega}{\partial x_1}, \quad 
\{x_3,x_1\} = \frac{\partial \Omega}{\partial x_2},
\]
for some homogeneous polynomials $\Omega$ of $x_1, x_2, x_3$ of degree 3. Such Poisson structures are called \textit{Jacobian}. In the literature, such Poisson structures are alternatively referred to as \textit{unimodular}, as in \cite[Proposition 2.6]{LWW}. The homogeneous polynomial $\Omega$ were classified into 9 subclasses up to some scalar in \cite{BW}:
\begin{align*}
&(1)\hspace{.1cm}x_1^3, \quad (2)\hspace{.1cm}x_1^2x_2, \quad (3)\hspace{.1cm}2x_1x_2x_3, \quad (4)\hspace{.1cm}x_1^2x_2 + x_1x_2, \quad (5)\hspace{.1cm}x_1^3 + x_1^2x_3, \quad (6)\hspace{.1cm}x_1^3 + x_1^2x_3 + x_2x_3,\\
&(7)\hspace{.1cm}\frac{1}{3}(x_1^3 + x_2^3 + x_3^3) + \lambda x_1x_2x_3 \hspace{.1cm}(\lambda^3 \neq -1),\quad (8)\hspace{.1cm}x_1^3 + x_1^2x_2 + x_1x_2x_3, \quad (9)\hspace{.1cm}x_1^2x_3 + x_1^2x_2.
\end{align*}

In the remaining part of this section, we will review the necessary concepts and tools required for the study of the invariants of these Poisson structures.

\smallskip

Let $P, Q$ be Poisson algebras. A map $\phi: P \to Q$ is called a \textit{Poisson homomorphism} if $\phi$ is an algebra homomorphism and a Lie algebra homomorphism. 

\smallskip 

Let $P = \Bbbk[x_1, \cdots, x_n]$ be a Poisson algebra under the standard grading. A \textit{graded Poisson automorphism} of $P$ is a bijective Poisson homomorphism $\phi: P \to P$ such that $\phi(P_i) = P_i$ for all $i \geq 0$. The graded Poisson automorphism group of $P$ will be denoted as $\text{PAut}_{\text{gr}}(P)$. A \textit{Poisson reflection} of $P$ is a finite-order graded Poisson automorphism $\phi: P \to P$ such that $\phi\big\vert_{P_1}$ has the following eigenvalues: $\underbrace{1,\cdots,1}_{n-1}, \xi$, for some primitive $m$th root of unity $\xi$. The set consisting of all Poisson reflections of $P$ will be denoted as $\text{PR}(P)$.

\smallskip

Let $P = \Bbbk[x_1, \cdots, x_n]$ be a Poisson algebra. Its Poisson enveloping algebra $U(P)$ is necessarily unique with respect to the universal property described in the introduction and can be described by an explicit set of generators and relations as follows:
\begin{thm}
\cite{OPS} Let $P = \Bbbk[x_1, \cdots, x_n]$ be a Poisson algebra. Then $U(P)$ is the free $\Bbbk$-algebra generated by $x_1, \cdots, x_n, y_1, \cdots, y_n$, subject to the following relations:
\end{thm}
\begin{enumerate}
    \item $[x_i, x_j] = 0$,
    
    \item $[y_i, y_j] = \displaystyle{\sum_{k=1}^{n}\frac{\partial \{x_i,x_j\}}{\partial x_k}y_k}$,

    \item $[y_i, x_j] = \{x_i, x_j\}$,
\end{enumerate}
for all $1 \leq i, j \leq n$, and $\alpha, \beta$ are defined as the follows:
\begin{alignat*}{3}
    \alpha: \hspace{.1cm} &P \to U(P), \quad \beta: P \to U(P)\\
    &f \hspace{.05cm} \mapsto f, \hspace{1.57cm} f \mapsto \sum_{k=1}^{n}\frac{\partial f}{\partial x_k}y_k.
\end{alignat*}

The salience of Poisson enveloping algebras within the research of invariant theory of Poisson algebras lies in its role as one of the two intermediary links connecting Poisson algebras to Artin-Schelter regular algebras. Let $P = \Bbbk[x_1, \cdots, x_n]$ be a quadratic Poisson algebra. Then $U(P)$ is Artin-Schelter regular \cite[Corollary 1.5]{LWZ2} and satisfies a range of preferred qualities:
\begin{enumerate}
    \item $U(P)$ is Noetherian \cite[Proposition 9]{O}.

    \item $U(P)$ admits a Poincaré-Birkhoff-Witt basis $\{x_1^{i_1}\cdots x_n^{i_n}y_1^{j_1} \cdots y_n^{j_n}: i_r, j_s \geq 0\}$ \cite[Theorem 3.7]{OPS}. 
    
    \item $U(P)$ has global dimension $2n$ \cite[Proposition 2.1]{BZ}.
    
    \item The Hilbert series $h_{U(P)}(t) = \displaystyle{\frac{1}{(1-t)^{2n}}}$ \cite[Lemma 5.4]{GVW}.
\end{enumerate}

Because $U(P)$ is an Artin-Schelter regular algebra, we can extend the commutative concepts of ``reflections" and ``determinant" to it.

\smallskip

Let $A$ be a connected $\mathbb{N}$-graded, locally finite $\Bbbk$-algebra and let $\phi$ be a graded automorphism of $A$. The \textit{trace series} of $\phi$ is $\text{Tr}_{A}(\phi,t) = \displaystyle{\sum_{i = 0}^{\infty}\text{tr}(\phi\big\vert_{A_i})t^i}$. In particular, if $\phi = \text{id}_{A}$, we recover the Hilbert series: $\text{Tr}_{A}(\text{id}_{A},t) = h_{A}(t)$.

\smallskip

A practical application of the trace series is the computation of the Hilbert series of invariant subalgebras:
\begin{thm}
(Molien's Theorem) Let $A$ be a connected $\mathbb{N}$-graded, locally finite $\Bbbk$-algebra and let $G \subseteq \text{Aut}_{\text{gr}}(A)$ be a finite subgroup. Then
\[
h_{A^G}(t) = \frac{1}{|G|}\sum_{\phi \in G}\text{Tr}_{A}(\phi,t).
\]
\end{thm}

\smallskip

Let $A$ be an Artin-Schelter regular algebra with Hilbert series $h_{A}(t) = \displaystyle{\frac{1}{(1-t)^{n}f(t)}}$ for some $f(t)$ satisfying $f(1) \neq 0$. A \textit{quasi-reflection} of $A$ is a finite-order graded automorphism $\phi: A \to A$ such that $\text{Tr}_{A}(\phi,t) = \displaystyle{\frac{1}{(1-t)^{n-1}g(t)}}$ for some $g(t)$ satisfying $g(1) \neq 0$. In particular, if $A$ is a commutative polynomial ring, we recover the notion of a classical reflection: $\phi\big\vert_{A_1}$ has eigenvalues $\underbrace{1, \cdots, 1}_{n-1}, \xi$ for some primitive $m$th root of unity $\xi$. If $A$ is a noncommutative quantum polynomial ring generated in degree 1, such as the Poisson enveloping algebra of a quadratic Poisson algebra, \cite[Theorem 3.1]{KKZ} proves that a quasi-reflection $\phi$ of $A$ necessarily takes one of the following form:
\begin{itemize}
    \item $\phi\big\vert_{A_1}$ has eigenvalues $\underbrace{1, \cdots, 1}_{n-1}, \xi$ for some primitive $m$th root of unity $\xi$,
    
    \item $\phi$ has order 4 and $\phi\big\vert_{A_1}$ has eigenvalues $\underbrace{1, \cdots, 1}_{n-2}, i, -i$.
\end{itemize}

\smallskip

We now review the notion of homological determinant, as introduced by Jing and Zhang
\cite{JZ}. Let $A$ be an Artin-Schelter Gorenstein algebra of injective dimension $n$ and let $\phi$ be a graded automorphism of $A$. The $\phi$-linear map $H_{{A}_{\geq 1}}^n(A) \to \hspace{0.01in} ^{\sigma}H_{{A}_{\geq 1}}^n(A)$ necessarily equals to the product of a nonzero scalar $c$ and the Matlis dual map $(\phi^{-1})'$. The \textit{homological determinant} of $\phi$ to be $\text{hdet} \phi = \displaystyle{\frac{1}{c}}$. In particular, if $A$ is a commutative polynomial ring, the homological determinant $\text{hdet} \phi$ coincides with $\det \phi\big\vert_{A_1}$.

\smallskip

The introduction of the homological determinant, as one might anticipate, aims to furish a noncommutative counterpart to the Watanabe Theorem:

\begin{thm}
(Non-commutative Watanabe Theorem, \cite[Theorem 3.3]{JZ}) Let $A$ be a Noetherian, Artin-Schelter Gorenstein $\Bbbk$-algebra and let $H \subseteq \text{GrAut}(A)$ be a finite subgroup. If $\text{hdet} h = 1$ for all $h \in H$, then $A^H$ is Artin-Schelter Gorenstein.
\end{thm}

\

\section{Preliminary Lemmas}

In this section, we will present all the requisite lemmas for proving a variant of the Shephard-Todd-Chevalley Theorem for quadratic unimodular Poisson structures on $\Bbbk[x_1, x_2, x_3]$, as well as a variant of the Shephard-Todd-Chevalley Theorem and a variant of the Watanabe Theorem for Poisson enveloping algebras of quadratic Poisson structures on $\Bbbk[x_1, \cdots, x_n]$. We will commence with an equation for the classification of graded Poisson automorphisms of these Poisson structures.
\smallskip

Let $P = \Bbbk[x_1, \cdots, x_n]$ be a Poisson algebra under the standard grading and let $\phi \in \text{PAut}_{\text{gr}}(P)$. The graded Poisson automorphism $\phi$ is uniquely determined by its action on $P_1 = \displaystyle{\bigoplus_{i=1}^{n}\Bbbk x_i}$, allowing it to be represented as an invertible $n \times n$ matrix: $\phi = \begin{bmatrix}a_{ij}\end{bmatrix}_{1 \leq i, j \leq n}$ for some $a_{ij} \in \Bbbk$. For $i \leq j$,
\begin{align}
\phi(\{x_i, x_j\}) = \{\phi(x_i), \phi(x_j)\} = \bigg\{\sum_{k=1}^{n}a_{ik}x_k, \sum_{j=1}^{n}a_{jl}x_l\bigg\} = 
\sum_{1 \leq k < l \leq n}(a_{ik}a_{jl} - a_{il}a_{jk})\{x_k, x_l\}. 
\end{align}

\smallskip

When investigating invariant subalgebras, a common objective is to ascertain whether two Poisson structures are isomorphic. Generally, establishing the non-isomorphism between a Poisson algebra $P$ and another Poisson algebra $Q$, by demonstrating the absence of suitable mappings, can be a challenging task. In this paper, we primarily rely on the following two lemmas to differentiate Poisson algebras fro their invariant subalgebras.

\smallskip

The first lemma asserts that unimodularity can serve as a distinguishing characteristic for Poisson algebras.

\begin{lem}
Let $P = \Bbbk[x_1, \cdots, x_n]$, $Q = \Bbbk[y_1, \cdots, y_n]$ be Poisson algebras. If $P$ is unimodular and $Q$ is non-unimodular, then $P$ is not isomorphic to $Q$ as Poisson algebras.
\end{lem}

\begin{proof}
Let $\underline{m}_P, \underline{m}_Q$ be the modular derivations of $P, Q$, respectively. Suppose that $P$ is isomorphic to $Q$ as Poisson algebras through an isomorphism $\phi$. Given that the modular derivation $\underline{m}_Q$ is independent of the generators \cite[page 6]{TWZ}, we may compute $\underline{m}_Q$ with respect to the generators $\{\phi(x_1), \cdots, \phi(x_n)\}$, instead of the conventional generators $\{y_1, \cdots, y_n\}$. For all $f \in Q$,
\[
\underline{m}_Q(f) = \sum_{i=1}^{n}\frac{\partial \{\phi(x_i), \phi(g)\}}{\partial \phi(x_i)} = \sum_{i=1}^{n}\frac{\partial \phi\{x_i, g\}}{\partial \phi(x_i)} = \phi\left(\sum_{i=1}^{n}\frac{\partial \{x_i, g\}}{\partial x_i}\right) = \phi(m_P(g)) = 0,
\]
for some $g \in P$. By definition, the Poisson algebra $Q$ is unimodular, a contradiction.
\end{proof}

\smallskip

The second lemma, while less general in scope compared to the former one, have nonetheless been proven to be highly useful in our paper. 

\begin{lem}
Let $P = \Bbbk[x_1, x_2, x_3]$ be a quadratic Poisson algebra. If $Q = \Bbbk[y_1, y_2, y_3]$ is a Poisson algebra satisfying the following conditions:
\begin{enumerate}[label = (\arabic*)]
    \item there exists a pair $(i, j) \in \{1,2,3\}^{\oplus 2}$ such that $\{y_i, y_j\} = f(y_1, y_2) - r y_3$ for some homogeneous $f$ with respect to the grading $\deg(y_1) = \deg(y_2) = 1$, and some $r\in \Bbbk^{\times}$,

    \item the bracket $\{y_k, y_l\}$ is a scalar multiple of a single monomial in $y_1$, $y_2$, $y_3$, for all $(k, l) \in \{1,2,3\}^{\oplus 2} \backslash \{(i, j)\}$,
\end{enumerate}
then $P$ is not isomorphic to $Q$ as Poisson algebras.
\end{lem}

\begin{proof}
Suppose, for the sake of contradiction, that $P$ and $Q$ are isomorphic as Poisson algebras. Observe that such a Poisson isomorphism passes to a $\Bbbk$-algebra isomorphism:

\[
P/(\{P, P\}) \xrightarrow{\sim} Q/(\{Q, Q\})
\]

For the $\Bbbk$-algebra $P$, the ideal $(\{P, P\})$ is generated by $\{x_i,x_j\}$, where $1 \leq i, j \leq 3$, because:
\[
\{g, h\} = \sum_{1 \leq i, j \leq 3}\frac{\partial g}{\partial x_i}\frac{\partial h}{\partial x_j}\{x_i, x_j\},
\]
for all $g, h \in P$. The quotient algebra $P/(\{P, P\})$ is a connected $\mathbb{N}$-graded algebra that is finitely generated in degree 1, as the Poisson algebra $P$ is quadratic and consequently the ideal $(\{P, P\})$ admits a set of homogeneous generators of degree 2.

For the $\Bbbk$-algebra $Q$, by the same reasoning, the ideal $\{Q, Q\}$ is generated by $\{y_i, y_j\}$, where $1 \leq i, j \leq 3$. Based on the assumptions, we can write
\[
\{y_a, y_b\}_Q = f(y_1, y_2) - r y_1, \quad \{y_b, y_c\}_Q = sy_1^{\alpha_1}y_2^{\alpha_2}y_3^{\alpha_3}, \quad \{y_c, y_a\}_Q = ty_1^{\beta_1}y_2^{\beta_2}y_3^{\beta_3},
\]
for some $\{a,b,c\}$ that is a relabeling of $\{1,2,3\}$, $s, t \in \Bbbk$, and $\alpha_1, \alpha_2, \alpha_3, \beta_1, \beta_2, \beta_3 \geq 0$. The quotient algebra
\begin{align*}
Q/(\{Q,Q\}) =& \Bbbk[y_1,y_2,y_3]/(f - ry_1, sy_1^{\alpha_1}y_2^{\alpha_2}y_3^{\alpha_3}, ty_1^{\beta_1}y_2^{\beta_2}y_3^{\beta_3})\\
\cong& \Bbbk[y_1,y_2]/(\frac{s}{r^{\alpha_3}}y_1^{\alpha_1}y_2^{\alpha_2}f^{\alpha_3}, \frac{t}{r^{\beta_3}}y_1^{\beta_1}y_2^{\beta_2}f^{\beta_3})
\end{align*}
is a connected $\mathbb{N}$-graded algebra that is finitely generated in degree 1, as $f$ is a homogeneous polynomial with respect to the grading $\deg(y_1) = \deg(y_2) = 1$. 

Based on the preceding argument, \cite[Lemma 4.1]{GW2} applies. However, 
\[
\dim_{\Bbbk}\big(P/(\{P,P\})\big) = 3 \neq 2 \geq \dim_{\Bbbk}\big(Q/(\{Q,Q\})\big),
\]
a contradiction.
\end{proof}

\smallskip

Next, we transition to the examination of the invariants of Poisson enveloping algebras. Let $P = \Bbbk[x_1,\cdots,x_n]$ be a quadratic Poisson algebra and let $U(P)$ be its Poisson enveloping algebra. Let $\phi \in \text{PAut}_{\text{gr}}(P)$. As in \cite[Lemma 5.1]{GVW}, let $\widetilde{\phi}$ denote the induced graded algebra automorphism of $U(P)$ defined as follows:
\begin{align*}
    \widetilde{\phi}(x_i) = \phi(x_i), \quad \widetilde{\phi}(y_i) = \sum_{j=1}^{n}\frac{\partial g(x_i)}{\partial x_j}y_j,
\end{align*}
for all $1 \leq i \leq n$. This is the natural action to consider for the following reasons:

\begin{lem}
Let $P = \Bbbk[x_1, \cdots, x_n]$ be a quadratic Poisson algebra. Suppose that $\phi$ is a graded Poisson automorphism of $P$. Then there exists a unique graded automorphism $\widetilde{\phi}$ on the Poisson enveloping algebra $U(P)$ such that
\[
\alpha \circ \phi = \widetilde{\phi} \circ \alpha, \quad \beta \circ \phi = \widetilde{\phi} \circ \beta.
\]
\end{lem}
\begin{proof}
Define $\widetilde{\phi}: U(P) \to U(P)$ as follows: $\widetilde{\phi}(x_i) = \phi(x_i)$ and $\widetilde{\phi}(y_i) = \displaystyle{\sum_{j=1}^{n}\frac{\partial \phi(x_i)}{\partial x_j}y_j}$, for all $1 \leq i \leq n$. For commutativity, $\alpha \circ \phi = \widetilde{\phi} \circ \alpha$; 
\[
\beta \circ \phi(x_i) = \sum_{j=1}^{n}\frac{\partial \phi(x_i)}{\partial x_j}y_j, \quad
\widetilde{\phi} \circ \beta(x_i) = \widetilde{\phi}(y_i) = \sum_{j=1}^{n}\frac{\partial \phi(x_i)}{\partial x_j}y_j,
\]
for all $1 \leq i \leq n$, and therefore $\beta \circ \phi = \widetilde{\phi} \circ \beta$. Suppose that $\widetilde{\widetilde{\phi}}: U(P) \to U(P)$ is another graded automorphism such that $\alpha \circ \phi = \widetilde{\widetilde{\phi}} \circ \alpha$ and $\beta \circ \phi = \widetilde{\widetilde{\phi}} \circ \beta$. From the commutativity $\alpha \circ \phi = \widetilde{\widetilde{\phi}} \circ \alpha$, we have $\widetilde{\widetilde{\phi}}(x_i) = \phi(x_i) = \widetilde{\phi}(x_i)$, for all $1 \leq i \leq n$. From the commutativity $\beta \circ \phi = \widetilde{\widetilde{\phi}} \circ \beta$, we have
\[
\widetilde{\widetilde{\phi}}(y_i) = \widetilde{\widetilde{\phi}} \circ \beta(x_i) = \beta \circ \phi(x_i) = \widetilde{\phi} \circ \beta(x_i) = \widetilde{\phi}(y_i),
\]
for all $1 \leq i \leq n$. Since $\widetilde{\widetilde{\phi}}$ and $\widetilde{\phi}$ agree on the generators of $U(P)$, the graded automorphism $\widetilde{\widetilde{\phi}}$ coincides with the graded automorphism $\widetilde{\phi}$.
\end{proof}

Retain the above notations. Suppose that $G$ is a subgroup of the graded Poisson automorphism of $P$. By Lemma 4.1, we can construct a subgroup $\widetilde{G} = \{\widetilde{\phi}: \phi \in G\}$ of the graded automorphism group of $U(P)$. It is natural to ask: is $\widetilde{G}$ isomorphic to $G$ as groups? Once again, the answer is affirmative.

\begin{lem}
Let $P = \Bbbk[x_1,\cdots,x_n]$ be a quadratic Poisson algebra. Suppose that $G$ is a subgroup of the graded Poisson automorphism group of $P$ and $\widetilde{G} = \{\widetilde{\phi}: \phi \in G\}$ is the corresponding subgroup of the graded automorphism group of $U(P)$. Then $\widetilde{G}$ is isomorphic to $G$ as groups. 
\end{lem}
\begin{proof}
Define $G \to \widetilde{G}$ by $\phi \mapsto \widetilde{\phi}$. First, we claim that this mapping is a group homomorphism. Let $\phi_1, \phi_2 \in G$. It is clear that $\widetilde{\phi_1}\widetilde{\phi_2}$ and $\widetilde{\phi_1\phi_2}$ agree on the generators $x_1, \cdots, x_n$. In the meantime, on the generators $y_1, \cdots, y_n$,
\begingroup
\allowdisplaybreaks
\begin{align*}
\widetilde{\phi_1}\widetilde{\phi_2}(y_i) &= \widetilde{\phi_1}(\widetilde{\phi_2}(y_i)) \hspace{5.45cm}\\
&= \widetilde{\phi_1}\left(\sum_{j=1}^{n}\frac{\partial \phi_2(x_i)}{\partial x_j}y_j\right)\\
&= \sum_{j=1}^{n}\phi_1\left(\frac{\partial \phi_2(x_i)}{\partial x_j}\right)\left(\sum_{k=1}^{n}\frac{\partial \phi_1(x_j)}{\partial x_k}y_k\right)\\
&= \sum_{j=1}^{n}\sum_{k=1}^{n}\phi_1\left(\frac{\partial \phi_2(x_i)}{\partial x_j}\right)\left(\frac{\partial \phi_1(x_j)}{\partial x_k}y_k\right)\\
&= \sum_{j=1}^{n}\sum_{k=1}^{n}\frac{\partial \phi_1\left(\phi_2(x_i)\right)}{\partial \phi_1(x_j)}\frac{\partial \phi_1(x_j)}{\partial x_k}y_k\\
&= \sum_{k=1}^{n}\frac{\partial \phi_1\left(\phi_2(x_i)\right)}{\partial x_k}y_k\\
&= \widetilde{\phi_1\phi_2}(y_i),
\end{align*}
\endgroup
for all $1 \leq i \leq n$, in which the fifth equality follows from the commutativity of the following diagram:
\[
\begin{tikzcd}
\Bbbk[x_1, \cdots, x_n] \arrow[rr, "\frac{\partial}{\partial x_j}"] \arrow[dd,"\phi_1",swap] & &\Bbbk[x_1, \cdots, x_n] \arrow[dd,"\phi_1"]\\
\\
\Bbbk[x_1, \cdots, x_n] \arrow[rr, "\frac{\partial}{\partial \phi_1(x_j)}",swap] & &\Bbbk[x_1, \cdots, x_n]
\end{tikzcd}
\]

For injectivity, suppose that $\widetilde{\phi_1} = \widetilde{\phi_2}$. On the generators $x_1, \cdots, x_n$, $\phi_1(x_i) = \widetilde{\phi_1}(x_i) = \widetilde{\phi_2}(x_i) = \phi_2(x_i)$, for all $1 \leq i \leq n$. Consequently, $\phi_1 = \phi_2$. Finally, surjectivity follows easily from the construction.
\end{proof}

\smallskip

The subsequent lemma examines the eigenvalues of the induced automorphisms on the Poisson enveloping algebra $U(P)$.

\begin{lem}
Let $P = \Bbbk[x_1, \cdots, x_n]$ be a quadratic Poisson algebra. Let $\phi$ be a graded Poisson automorphism of $P$ and let $\widetilde{\phi}$ be the corresponding graded automorphism of $U(P)$. Suppose that $\phi\big\vert_{P_1}$ has eigenvalues $\lambda_1, \cdots, \lambda_m$, with multiplicity $c_1, \cdots, c_m$, respectively. Then $\widetilde{\phi}\big\vert_{U(P)_1}$ has eigenvalues $\lambda_1, \cdots, \lambda_m$, with multiplicity $2c_1, \cdots, 2c_m$, respectively.
\end{lem}
\begin{proof}
The Poisson enveloping algebra $U(P)$ is a quadratic $\Bbbk$-algebra generated by $x_1, \cdots, x_n$, $y_1, \cdots, y_n$. Fix $1 \leq i \leq m$. Let $\{v_{i,1}, \cdots, v_{i,c_i}\}$ be a basis for the eigenspace of $\lambda_i$ in $P_1$. By calculation,
\[
\widetilde{\phi}(v_{i,j}) = \phi(v_{i,j}) = \lambda_i v_{i,j},
\]
\begin{align*}
\widetilde{\phi}(\sum_{k=1}^{n}\frac{\partial \phi(v_{i,j})}{\partial x_k}y_k) =
\sum_{k=1}^{n}\phi\left(\frac{\partial (\lambda_i v_{i,j})}{\partial x_k}\right)\widetilde{\phi}(y_k) =&
\lambda_i \sum_{k=1}^{n}\phi(\frac{\partial v_{i,j}}{\partial x_k})\widetilde{\phi}(y_k)\\ =& 
\lambda_i \sum_{k=1}^{n}\sum_{l=1}^{n}\frac{\partial \phi(v_{i,j})}{\partial \phi(x_k)}\frac{\partial \phi(x_k)}{\partial x_l}y_l = 
\lambda_i \sum_{l=1}^{n}\frac{\partial \phi(v_{i,j})}{\partial x_l}y_l.
\end{align*}
for all $1 \leq j \leq c_i$. Given that $\widetilde{\phi}$ is a graded automorphism, the vectors 
\[
\widetilde{\phi}(v_{i,1}), \cdots, \widetilde{\phi}(v_{i,c_i}), \displaystyle{\widetilde{\phi}(\sum_{k=1}^{n}\frac{\partial \phi(v_{i,j})}{\partial x_k}y_k}), \cdots, \displaystyle{\widetilde{\phi}(\sum_{k=1}^{n}\frac{\partial \phi(v_{c_i})}{\partial x_k}y_k})
\]
are linearly independent in $U(P)_1$. Since $\displaystyle{\sum_{i=1}^{m}c_i} = n$, $\displaystyle{\sum_{i=1}^{m}2c_i} = 2n$. In light of the fact that $\displaystyle{\sum_{i=1}^{m}c_i} = n$, it follows that $\displaystyle{\sum_{i=1}^{m}2c_i} = 2n$. This assertation implies that the set
\[
\left\{\widetilde{\phi}(v_{i,1}), \cdots, \widetilde{\phi}(v_{i,c_i}), \displaystyle{\widetilde{\phi}(\sum_{k=1}^{n}\frac{\partial \phi(v_{i,j})}{\partial x_k}y_k}), \cdots, \displaystyle{\widetilde{\phi}(\sum_{k=1}^{n}\frac{\partial \phi(v_{c_i})}{\partial x_k}y_k})\right\}_{1 \leq i \leq m}
\]
forms an eigenbasis for $U(P)_1$. In particular, the multiplicity of $\lambda_i$ in $U(P)_1$ equals to twice of multiplicity of $\lambda_i$ in $P_1$.
\end{proof}

\

\section{A Variant of The Shephard-Todd-Chevalley Theorem}
In this section, we prove a variant of the Shephard-Todd-Chevalley Theorem for unimodular quadratic Poisson structures on $\Bbbk[x_1, x_2, x_3]$, stated as the follows:
\begin{thm}
Let $P = \Bbbk[x_1, x_2, x_3]$ be a unimodular quadratic Poisson algebra and let $G \subseteq \text{PAut}_{\text{gr}}(P)$ be a finite subgroup. Then the invariant subalgebra $P^G$ is isomorphic to $P$ as Poisson algebras if and only if $P$ is trivial. 
\end{thm}

\smallskip

Prior to embarking on the proof of Theorem 3.1, we present a case-by-case study of all the quadratic unimodular Poisson structures on $\Bbbk[x_1, x_2, x_3]$.

\smallskip

\subsection{\texorpdfstring{$\Omega_1 = x_1^3, \quad \{x_1, x_2\} = 0, \quad \{x_2, x_3\} = 3x_1^2, \quad \{x_3, x_1\} = 0$}{Case 1}}

\begin{mylem}
\hspace{.1cm}
\begin{center}
\renewcommand{\arraystretch}{1.25}
\begin{tabular}{ |p{8cm}|p{8cm}|} 
\hline
\centering ${\textbf{PAut}_{{\textbf{gr}}}\boldsymbol{(P)}}$ &
\centering $\textbf{PR(}\boldsymbol{P}\textbf{)}$ \tabularnewline
\hline
\hline
\centering $\left\lbrace
\begin{bmatrix}
\pm \sqrt{bf-ce} & 0 & 0\\
a & b & c\\
d & e & f
\end{bmatrix}: bf \neq ce \right\rbrace$ &
\centering $\left\{
\begin{bmatrix}-1 & 0 & 0\\a & 1 & 0\\d & 0 & 1\end{bmatrix}
\right\}$
\tabularnewline
\hline
\end{tabular}
\end{center}
\end{mylem}

\begin{proof}
Let $\phi \in \text{PAut}_{\text{gr}}(P)$. Apply (2.1), we have the following system of equations, with redundant equations omitted:
\begin{enumerate}[label = (\arabic*)]
    \item $a_{11}^2 = a_{22}a_{33}-a_{23}a_{32}.$
    
    \item $a_{12}^2 = 0$.
    
    \item $a_{13}^2 = 0$.
\end{enumerate}

These relations simplify to $a_{11} = \pm \sqrt{a_{22}a_{33}-a_{23}a_{32}} \neq 0$, $a_{12} = a_{13} = 0$. In conclusion,
\[
\text{PAut}_{\text{gr}}(P) = \left\lbrace
\begin{bmatrix}
\pm \sqrt{bf-ce} & 0 & 0\\
a & b & c\\
d & e & f
\end{bmatrix}: bf \neq ce \right\rbrace.
\]

\smallskip

Again, let $\phi \in \text{PAut}_{\text{gr}}(P)$. Its eigenvalues are $\lambda_1 = \pm \sqrt{bf-ce}$, $\lambda_2, \lambda_3 = \displaystyle{\frac{b+f\pm\sqrt{(b-f)^2+4ce}}{2}}$. Notice that $\lambda_2\lambda_3 = \lambda_1^2$. If $\phi$ is a Poisson reflection, $\{\lambda_1, \lambda_2, \lambda_3\} = \{1,1,\xi\}$ for some primitive root of unity $\xi$. If $\lambda_1 = 1$, then $\lambda_2\lambda_3 = 1$, contradicting to $\{\lambda_2, \lambda_3\} = \{1, \xi\}$. If $\lambda_1 = \xi$, then $\lambda_2\lambda_3 = \xi^2$, contradicting to $\{\lambda_2, \lambda_3\} = \{1,1\}$ unless $\xi = -1$. In that case, $\lambda_1 = \pm \sqrt{bf-ce} = -1$, $\lambda_2 + \lambda_3 = b + f = 2$, and $\phi\big\vert_{P_1}$ takes the form
$\begin{brsm}-1 & 0 & 0\\a & b & c\\d & e & 2-b\end{brsm}$ subject to the constraint $b(2-b) - ce = 1$. Upon computation, it is found that the (2,3)-entry of the matrix, when raised to the $n$th power, is $nc$. If $\phi$ has finite order, then $c = 0$. Given that $b(2-b) = 1$, it follows that $b = 1$ and $\phi\big\vert_{P_1}$ takes a simpler form 
$\begin{brsm}-1 & 0 & 0\\a & 1 & 0\\d & e & 1\end{brsm}$. Again, upon computation, it is found that the (3,2)-entry of the matrix, when raised to the $n$th power, is $ne$. If $\phi$ has finite order, then $e = 0$. Upon substituting this value, the resulting matrix has finite order 2. In conclusion,
\[
\text{PR}(P) = \left\{
\begin{bmatrix}-1 & 0 & 0\\a & 1 & 0\\d & 0 & 1\end{bmatrix}
\right\}.
\]
\end{proof}

\smallskip

\begin{mylem}
If $G$ is a finite Poisson reflection group of $P$, then the invariant subalgebra $P^G$ is not isomorphic to $P$ as Poisson algebras. 
\end{mylem}
\begin{proof}
If $G$ contains two distinct Poisson reflections: 
$\phi_1\big\vert_{P_1} = \begin{brsm}-1 & 0 & 0\\a_1 & 1 & 0\\d_1 & 0 & 1\end{brsm}$, $\phi_2\big\vert_{P_1} = \begin{brsm}-1 & 0 & 0\\a_2 & 1 & 0\\d_2 & 0 & 1\end{brsm}$,
then the product $(\phi_1\phi_2)\big\vert_{P_1} = \begin{brsm}1 & 0 & 0\\a_2-a_1 & 1 & 0\\d_2-d_1 & 0 & 1\end{brsm}$, a matrix of infinite order unless $a_2 = a_1$ and $d_2 = d_1$. Consequently, we may assume that $G$ is cyclic generated: $G = \begin{brsm}-1 & 0 & 0\\a & 1 & 0\\d & 0 & 1\end{brsm} \cong \mathbb{Z}_2$. To compute the invariant subalgebra, we start by observing that the polynomials $y_1 = x_1^2$, $y_2 = \displaystyle{\frac{a}{2}}x_1 + x_2$, $y_3 = \displaystyle{\frac{d}{2}}x_1+x_3$ are algebraically independent and remain invariant under the action of $G$. Embed the $\Bbbk$-algebra generated by $y_1, y_2, y_3$ into the invariant subalgebra $P^{G}$. Utilizing Molien's Theorem to compute $h_{P^{G}}(t)$, we compare the Hilbert series of these two $\Bbbk$-algebras, and conclude that $P^G = \Bbbk[y_1, y_2, y_3]$. The Poisson structure on the invariant subalgebra $P^G$ is:
\[
\{y_1, y_2\} = 0, \quad \{y_2, y_3\} = 3y_1, \quad \{y_3, y_1\} = 0.
\]
By invoking Lemma 2.1, it becomes evident that the Poisson algebras $P$ and $P^G$ are non-isomorphic Poisson algebras.
\end{proof}

\smallskip
\subsection{\texorpdfstring{$\Omega_2 = x_1^2x_2, \quad \{x_1, x_2\} = 0, \quad \{x_2, x_3\} = 2x_1x_2, \quad \{x_3, x_1\} = x_1^2$}{Case 2}}

\begin{mylem}
\hspace{.1cm}
\begin{center}
\renewcommand{\arraystretch}{1.25}
\begin{tabular}{ |p{8cm}|p{8cm}|} 
\hline
\centering ${\textbf{PAut}_{{\textbf{gr}}}\boldsymbol{(P)}}$ &
\centering $\textbf{PR(}\boldsymbol{P}\textbf{)}$ \tabularnewline
\hline
\hline
\centering $\left\lbrace
\begin{bmatrix}
a & 0 & 0\\
0 & b & 0\\
c & d & a
\end{bmatrix}: a, b \neq 0\right\rbrace$ &
\centering $\left\{
\begin{bmatrix}1 & 0 & 0\\0 & \xi & 0\\0 & d & 1\end{bmatrix}
\right\}$
\tabularnewline
\hline
\end{tabular}
\end{center}
\end{mylem}

\begin{proof}
Let $\phi \in \text{PAut}_{\text{gr}}(P)$. Apply (2.1), we have the following system of equations, with redundant equations omitted:
\begin{enumerate}[label = (\arabic*)]
    \item $a_{13}a_{21} = a_{11}a_{23}.$
        
    \item ${2a_{11}a_{21} = a_{23}a_{31}-a_{21}a_{33}}$. 
    
    \item ${a_{11}^2 = a_{11}a_{33}-a_{13}a_{31}}.$
    
    \item ${a_{12}^2 = 0}.$
    
    \item ${a_{13}^2 = 0}.$
\end{enumerate}

By (4) and (5), $a_{12} = a_{13} = 0$ and $a_{11} \neq 0$. By substituting the variables in (1) and (3), $a_{23} = 0$ and $a_{33} = a_{11}$. Finally, it can be deduced from (2) that $a_{21} = 0$. In conclusion,
\[
\text{PAut}_{\text{gr}}(P) = \left\lbrace
\begin{bmatrix}
a & 0 & 0\\
0 & b & 0\\
c & d & a
\end{bmatrix}: a, b \neq 0\right\rbrace.
\]

\smallskip

Again, let $\phi \in \text{PAut}_{\text{gr}}(P)$. Its eigenvalues are $a, a, b$. If $\phi$ is a Poisson reflection, $a = 1$, $b = \xi$ for some primitive root of unity $\xi$, and $\phi\big\vert_{P_1}$ takes the form $\begin{brsm}1 & 0 & 0\\0 & \xi & 0\\c & d & 1\end{brsm}$. When raised to the $n$th power, the (3,1)-entry of the matrix $(\phi\big\vert_{P_1})^n$ is equal to $nc$, implying that $c = 0$. In the meantime, the (2,2)-entry of the matrix is equal to $\xi^n$ and the (3,2)-entry of the matrix is equal to $\displaystyle{d\frac{(\xi^n-1)}{\xi-1}}$. If $n$ is a multiple of the order of $\xi$, the matrix $(\phi\big\vert_{P_1})^n$ is equal to the identity matrix. In conclusion,
\[
\text{PR}(P) = \left\{
\begin{bmatrix}1 & 0 & 0\\0 & \xi & 0\\0 & d & 1\end{bmatrix}
\right\}.
\]
\end{proof}

\smallskip

\begin{mylem}
If $G$ is a finite Poisson reflection group of $P$, then the invariant subalgebra $P^G$ is not isomorphic to $P$ as Poisson algebras. 
\end{mylem}

\begin{proof}
If $G$ contains two non-commuting Poisson reflections: $\phi_1\big\vert_{P_1} = \begin{brsm}1 & 0 & 0\\0 & \xi_{n_1} & 0\\0 & d_1 & 1\end{brsm}$ and $\phi_2\big\vert_{P_1} = \begin{brsm}1 & 0 & 0\\0 & \xi_{n_2} & 0\\0 & d_2 & 1\end{brsm}$, where $\xi_{n_1}$ (resp. $\xi_{n_2}$) is a primitve $n_1$th (resp. $n_2$th) root of unity. Since $\phi_1\phi_2 \neq \phi_2\phi_1$, the product $(\phi_1\phi_2\phi_1^{n_1-1}\phi_2^{n_2-1})\big\vert_{P_1} = \begin{brsm}1 & 0 & 0\\0 & 1 & 0\\0 & d & 1\end{brsm}$ for some $d \neq 0$; however, such a matrix has infinite order, a contradiction. Therefore, $G$ is a finite abelian group. Keeping $\phi_1$ and $\phi_2$ as above, the commutativity $\phi_1\phi_2 = \phi_2\phi_1$ is equivalent to $d_2 = \displaystyle{\frac{\xi_{n_2}-1}{\xi_{n_1}-1}d_1}$. 

Decompose $G \cong \mathbb{Z}_{n_1} \times \cdots \times \mathbb{Z}_{n_m}$ for some $n_i \in \mathbb{N}$ satisfying $n_i | n_{i+1}$. According to the commutativity condition, we may take the cyclic generator of $\mathbb{Z}_{n_i}$ to be $\phi_i\big\vert_{P_1} = \begin{brsm}1 & 0 & 0\\0 & \xi_{n_i} & 0\\0 & \frac{\xi_{n_i}-1}{\xi_{n_1}-1}d_1 & 1\end{brsm}$ for some primitive $n_i$th root of unity $\xi_{n_i}$, for all $1 \leq i \leq m$. Define $S = \{\alpha \in \mathbb{N}^m: 1 \leq \alpha_i \leq n_i\}$. For $\alpha = (\alpha_i)_{1 \leq i \leq m} \in S$, we define $\xi^{\alpha} = \xi_{n_1}^{\alpha_1} \cdots \xi_{n_m}^{\alpha_m}$. According to the Molien's Theorem,
\[
h_{P^G}(t) = \frac{1}{n_1 \cdots n_m}\sum_{\alpha \in S}\frac{1}{(1-t)^2(1 - \xi^{\alpha}t)} = \frac{1}{(1-t)^2}\left(\frac{1}{n_1 \cdots n_m}\sum_{\alpha \in S}\frac{1}{(1 - \xi^{\alpha}t)}\right).
\]

Before computing $\displaystyle{\sum_{\alpha \in S}\frac{1}{(1 - \xi^{\alpha}t)}}$, we claim that $\{\xi^{\alpha}: \alpha \in S\}$ is precisely $n_1 \cdots n_{m-1}$ copies of $\{\xi_{n_m}^{\alpha_m}: 1 \leq \alpha_m \leq n_m\}$. To prove the claim, we proceed by induction. When $m = 2$, the set $\{\xi^{\alpha}: \alpha \in S\}$ contains the following elements:
\begin{alignat*}{8}
1 & \hspace{.2in} & \xi_{n_2} & \hspace{.2in} & \xi_{n_2}^{2} & \hspace{.2in} & \cdots & \hspace{.2in} & \xi_{n_2}^{n_2-1}\\
1 & \hspace{.2in} & \xi_{n_1}\xi_{n_2} & \hspace{.2in} & \xi_{n_1}\xi_{n_2}^{2} & \hspace{.2in} & \cdots & \hspace{.2in} & \xi_{n_1}\xi_{n_2}^{n_2-1}\\
\vdots & \hspace{.2in} & \vdots & \hspace{.2in} & \vdots & \hspace{.2in} & \ddots & \hspace{.2in} & \vdots\\
1 & \hspace{.2in} & \xi_{n_1}^{n_1-1}\xi_{n_2} & \hspace{.2in} & \xi_{n_1}^{n_1-1}\xi_{n_2}^{2} & \hspace{.2in} & \cdots & \hspace{.2in} & \xi_{n_1}^{n_1-1}\xi_{n_2}^{n_2-1}
\end{alignat*}
Each line can be realized as the image of a left multiplication map $l_{\xi_{n_1}^{\alpha_1}}: \langle \xi_{n_2} \rangle \to \langle \xi_{n_2} \rangle$, for all $0 \leq \alpha_1 \leq n_1-1$. Since $n_1 | n_2$, and consequently, $\xi_{n_1}^{\alpha_1}$ is an element of $\langle \xi_{n_2} \rangle$, each line is a permutation of the first line. This proves the base case: $\{\xi_{n_1}^{\alpha_1}\xi_{n_2}^{\alpha_2}| 1 \leq \alpha_1 \leq n_1, 1 \leq \alpha_2 \leq n_2\}$ is precisely $n_1$ copies of $\{\xi_{n_2}^{\alpha_2}| 1 \leq \alpha_2 \leq n_2\}$. Inductively, when $m$ is arbitrary, the set $\{\xi^{\alpha} | \alpha \in S\}$ contains the following elements:
\begin{alignat*}{8}
1 & \hspace{.2in} & E_{m-1} & \hspace{.2in} & E_{m-1}\xi_{n_m} & \hspace{.2in} & \cdots & \hspace{.2in} & E_{m-1}\xi_{n_m}^{n_m-1},
\end{alignat*}
where $E_{m-1} = \{\xi_{n_1}^{\alpha_1} \cdots \xi_{n_{m-1}}^{\alpha_{n_{m-1}}}: 1 \leq \alpha_i \leq n_i\}$. By the induction hypothesis, the set $\{\xi_{n_1}^{\alpha_1} \cdots \xi_{n_{m-1}}^{\alpha_{n_{m-1}}}: 1 \leq \alpha_i \leq n_i\}$ is $n_1n_2 \cdots n_{m-2}$ copies of $\{\xi_{n_{m-1}}^{\alpha_{m-1}} : 1 \leq \alpha_{m-1} \leq n_{m-1}\}$. Consequently, we can view the preceding lines of elements as $n_1n_2 \cdots n_{m-2}$ layers of the subsequent lines of elements:
\begingroup
\allowdisplaybreaks
\begin{alignat*}{8}
    1 & \hspace{.2in} & \xi_{n_{m}} & \hspace{.2in} & \xi_{n_m}^{2} & \hspace{.2in} & \cdots & \hspace{.2in} & \xi_{n_m}^{n_m-1}\\
    1 & \hspace{.2in} & \xi_{n_{m-1}}\xi_{n_m} & \hspace{.2in} & \xi_{n_{m-1}}\xi_{n_m}^{2} & \hspace{.2in} & \cdots & \hspace{.2in} & \xi_{n_{m-1}}\xi_{n_m}^{n_m-1}\\
    \vdots & \hspace{.2in} & \vdots & \hspace{.2in} & \vdots & \hspace{.2in} & \ddots & \hspace{.2in} & \vdots\\
    1 & \hspace{.2in} & \xi_{n_{m-1}}^{n_{m-1}-1}\xi_{n_m} & \hspace{.2in} & \xi_{n_{m-1}}^{n_{m-1}-1}\xi_{n_m}^{2} & \hspace{.2in} & \cdots & \hspace{.2in} & \xi_{n_{m-1}}^{n_{m-1}-1}\xi_{n_m}^{n_m-1}
\end{alignat*}
\endgroup

Again, from the induction hypothesis, the above lines of elements are $n_1n_2 \cdots n_{m-1}$ copies $\{\xi_{n_m}^{\alpha_m} | 1 \leq \alpha_m \leq n_m\}$ as claimed. We can compute the Hilbert series of the invariant subalgebra $P^G$ by employing the claim:

\begin{align*}
    h_{P^{G}}(t) =& \frac{1}{(1-t)^2}\bigg(\frac{n_1 \cdots n_{m-1}}{n_1 \cdots n_m}\sum_{1 \leq \alpha_m \leq n_m}{\frac{1}{(1 - \xi_{n_m}^{\alpha_m}t)}}\bigg)\\
    =& \frac{1}{(1-t)^2}\bigg(\frac{1}{n_m}\sum_{1 \leq \alpha_m \leq n_m}{\frac{1}{(1 - \xi_{n_m}^{\alpha_m}t)}}\bigg)\\
    =& \frac{1}{(1-t)^2(1-t^{n_m})}.
\end{align*}

Set $l = n_m$. Furthermore, set $y_1 = x_1$, $y_2 = d_1x_2 + (1-\xi_{n_1})x_3$, $y_3 = x_2^l$. The elements $y_1, y_2, y_3$ are three algebraically independent polynomials that are invariant under the action of $G$. Consequently, we may embed $\Bbbk[y_1, y_2, y_3]$ into the invariant subalgebra $P^G$. It is evident that the Hilbert series of $\Bbbk[y_1,y_2,y_3]$ is $\displaystyle{\frac{1}{(1-t)^2(1-t^l)}}$, and therefore, the embedding $\Bbbk[y_1, y_2, y_3] \hookrightarrow P^G$ is surjective because the cokernel has Hilbert series 0. Accordingly, we conclude that the invariant subalgebra $P^G = \Bbbk[y_1, y_2, y_3]$ and has the following Poisson structure:
\[
\{y_1, y_2\} = (\xi_{n_1}-1)y_1^2, \quad \{y_2,y_3\} = 2l(\xi_{n_1}-1)y_1y_3, \quad \{y_3, y_1\} = 0.
\]

Suppose that $P^G \cong P$ as Poisson algebras. Specifically, by applying (the contrapositive of) Lemma 2.1, the invariant subalgebra is unimodular. This implies that we can find a superpotential $\Omega \in \Bbbk[y_1, y_2, y_3]$ satisfying:
\[
\frac{\partial \Omega}{\partial y_3} = (\xi_{n_1}-1)y_1^2, \quad \frac{\partial \Omega}{\partial y_1} =2l(\xi_{n_1}-1)y_1y_3, \quad \frac{\partial \Omega}{\partial y_2} = 0.
\]
It is straightforward to verify no such $\Omega$ exists unless $l = 1$, a contradiction to $G$ being a Poisson reflection group. 
\end{proof}

\smallskip
\subsection{\texorpdfstring{$\Omega_3 = 2x_1x_2x_3, \quad \{x_1, x_2\} = 2x_1x_2, \quad \{x_2, x_3\} = 2x_2x_3, \quad \{x_3, x_1\} = 2x_1x_3$}{Case 3}}

\begin{mylem}
\hspace{.1cm}
\begin{center}
\renewcommand{\arraystretch}{1.25}
\begin{tabular}{ |p{8cm}|p{8cm}|} 
\hline
\centering ${\textbf{PAut}_{{\textbf{gr}}}\boldsymbol{(P)}}$ &
\centering $\textbf{PR(}\boldsymbol{P}\textbf{)}$ \tabularnewline
\hline
\hline
\centering $\left\lbrace
\begin{bmatrix}
a & 0 & 0\\
0 & b & 0\\
0 & 0 & c
\end{bmatrix}, 
\begin{bmatrix}
0 & a & 0\\
0 & 0 & b\\
c & 0 & 0
\end{bmatrix},
\begin{bmatrix}
0 & 0 & a\\
b & 0 & 0\\
0 & c & 0
\end{bmatrix}
: a, b, c \neq 0\right\rbrace$ &
\centering $\left\{
\begin{bmatrix}\xi & 0 & 0\\0 & 1 & 0\\0 & 0 & 1\end{bmatrix},
\begin{bmatrix}1 & 0 & 0\\0 & \xi & 0\\0 & 0 & 1\end{bmatrix}, 
\begin{bmatrix}1 & 0 & 0\\0 & 1 & 0\\0 & 0 & \xi\end{bmatrix}
\right\}$
\tabularnewline
\hline
\end{tabular}
\end{center}
\end{mylem}

\begin{proof}
Let $\phi \in \text{PAut}_{\text{gr}}(P)$. Apply (2.1): $a_{ik}a_{jk} = 0$ for all $i, j, j \in \{1,2,3\}, i \neq j$. It follows that $\phi$ may be represented as a scalar permutation matrix. In conclusion,
\[
\text{PAut}_{\text{gr}}(P) = \left\lbrace
\begin{bmatrix}
a & 0 & 0\\
0 & b & 0\\
0 & 0 & c
\end{bmatrix}, 
\begin{bmatrix}
0 & a & 0\\
0 & 0 & b\\
c & 0 & 0
\end{bmatrix},
\begin{bmatrix}
0 & 0 & a\\
b & 0 & 0\\
0 & c & 0
\end{bmatrix}
: a, b, c \neq 0\right\rbrace.
\]

\smallskip

Again, let $\phi \in \text{PAut}_{\text{gr}}(P)$. If $\phi\big\vert_{P_1} = \begin{brsm}
a & 0 & 0\\0 & b & 0\\0 & 0 & c\end{brsm}$, then $\phi$ is a Poisson reflection if and only if $\phi\big\vert_{P_1} = \begin{brsm}\xi & 0 & 0\\0 & 1 & 0\\0 & 0 & 1\end{brsm}$, $\begin{brsm}1 & 0 & 0\\0 & \xi & 0\\0 & 0 & 1\end{brsm}$, $\begin{brsm}1 & 0 & 0\\0 & 1 & 0\\0 & 0 & \xi\end{brsm}$ for some primitive root of unity $\xi$. If $\phi\big\vert_{P_1} = \begin{brsm}0 & a & 0\\0 & 0 & b\\c & 0 & 0\end{brsm}$ or $\begin{brsm}0 & 0 & a\\b & 0 & 0\\0 & c & 0\end{brsm}$, its eigenvalues $\lambda_1 = \sqrt[3]{abc}, \lambda_2 = \sqrt[3]{abc}\xi_3, \lambda_3 = \sqrt[3]{abc}{\xi_3^2}$, where $\xi_3$ is a primitive 3rd root of unity. If further $\phi$ is a Poisson reflection, then $\{\lambda_1, \lambda_2, \lambda_3\} = \{1, 1, \xi\}$ for some primitive root of unity $\xi$. If $\lambda_1 = 1$, then $\lambda_2\lambda_3 = \xi$, contradicting to $\lambda_2\lambda_3 = \lambda_1^2$. If $\lambda_1 = \xi$, then $\displaystyle{\frac{\lambda_3}{\lambda_2}} = \frac{1}{1} = 1$, contradicting to $\displaystyle{\frac{\lambda_3}{\lambda_2} = \frac{\sqrt[3]{abc}\xi_3^2}{\sqrt[3]{abc}\xi_3}} = \xi_3$. Consequently, such $\phi$ cannot be Poisson reflections. In conclusion,
\[
\text{PR}(P) = \left\{
\begin{bmatrix}\xi & 0 & 0\\0 & 1 & 0\\0 & 0 & 1\end{bmatrix},
\begin{bmatrix}1 & 0 & 0\\0 & \xi & 0\\0 & 0 & 1\end{bmatrix}, 
\begin{bmatrix}1 & 0 & 0\\0 & 1 & 0\\0 & 0 & \xi\end{bmatrix}
\right\}.
\]
\end{proof}

\smallskip

\begin{mylem}
If $G$ is a finite Poisson reflection group of $P$, then the invariant subalgebra $P^G$ is not isomorphic to $P$ as Poisson algebras. 
\end{mylem}
\begin{proof}
This case is addressed in \cite[Theorem 4.5]{GVW}. The conclusion is as follows: $P^G \cong P$ as Poisson algebras if and only if $G$ is trivial.
\end{proof}

\smallskip
\subsection{\texorpdfstring{$\Omega_4 = x_1^2x_2 + x_1x_2, \quad \{x_1, x_2\} = 0, \quad \{x_2, x_3\} = 2x_1x_2 + x_2^2, \quad \{x_3, x_1\} = x_1^2 + 2x_1x_2$}{Case 4}}

\begin{mylem}
\hspace{.1cm}
\begin{center}
\renewcommand{\arraystretch}{1.25}
\begin{tabular}{ |p{8cm}|p{8cm}|} 
\hline
\centering ${\textbf{PAut}_{{\textbf{gr}}}\boldsymbol{(P)}}$ &
\centering $\textbf{PR(}\boldsymbol{P}\textbf{)}$ \tabularnewline
\hline
\hline
\centering $\left\lbrace
\begin{bmatrix}0 & a & 0\\-a & -a & 0\\b & c & a\end{bmatrix},
\begin{bmatrix}0 & -a & 0\\-a & 0 & 0\\b & c & a\end{bmatrix},
\begin{bmatrix}-a & 0 & 0\\a & a & 0\\b & c & a\end{bmatrix},
\color{white}\right\rbrace$ &
\centering $\left\{
\begin{bmatrix}0 & -1 & 0\\-1 & 0 & 0\\b & b & 1\end{bmatrix},
\begin{bmatrix}-1 & 0 & 0\\1 & 1 & 0\\b & 0 & 1\end{bmatrix}, 
\begin{bmatrix}1 & 1 & 0\\0 & -1 & 0\\0 & c & 1\end{bmatrix}
\right\}$
\tabularnewline
\vspace{.00001cm}
\centering $\color{white}\left\rbrace\color{black}
\begin{bmatrix}-a & -a & 0\\a & 0 & 0\\b & c & a\end{bmatrix},
\begin{bmatrix}a & a & 0\\0 & -a & 0\\b & c & a\end{bmatrix}, 
\begin{bmatrix}a & 0 & 0\\0 & a & 0\\b & c & a\end{bmatrix}
: a \neq 0\right\rbrace\color{black}$ &
\tabularnewline
\hline
\end{tabular}
\end{center}
\end{mylem}

\begin{proof}
Let $\phi \in \text{PAut}_{\text{gr}}(P)$. Apply (2.1), we have the following system of equations, with redundant equations omitted:
\begin{enumerate}[label = (\arabic*)]
    
    \item $2a_{11}a_{21}+a_{21}^2 = a_{23}a_{31}-a_{21}a_{33}$.

    \item $2a_{12}a_{22}+a_{22}^2 = a_{22}a_{33}-a_{23}a_{32}$.
    
    \item $2a_{13}a_{23}+a_{23}^2 = 0$.
    
    \item $a_{11}a_{22}+a_{12}a_{21}+a_{21}a_{22} = a_{22}a_{33} - a_{23}a_{32} + a_{23}a_{31} - a_{21}a_{33}$.
    
    \item $a_{11}^2+2a_{11}a_{21} = a_{11}a_{33}-a_{13}a_{31}$.
    
    \item $a_{12}^2+2a_{12}a_{22} = a_{13}a_{32}-a_{12}a_{33}$.
    
    \item $a_{13}^2+2a_{13}a_{23} = 0$.
\end{enumerate}

It can be deduced from (3) and (7) that $a_{13} = a_{23} = 0$ and $a_{33} \neq 0$. Next, we will proceed with a case-by-case discussion.
\begin{itemize}
    \item Suppose that $a_{21} \neq 0$ and $a_{11} = 0$. From invertibility, $a_{12} \neq 0$. From (1), $a_{21} = -a_{33}$. If $a_{22} \neq 0$, the combination of (2) and (6) leads to $a_{12} = -a_{22} = a_{33}$ and the remaining equations are nullified. If $a_{22} = 0$, (4) states that $a_{12} = -a_{33}$ and the remaining equations are nullified. Consequently, we have two possible forms for $\phi$: 
    $\begin{brsm}0 & a & 0\\-a & -a & 0\\b & c & a\end{brsm}$, $\begin{brsm}0 & -a & 0\\-a & 0 & 0\\b & c & a\end{brsm}$, for some $a \neq 0$.

    \item Suppose that $a_{21} \neq 0$ and $a_{11} \neq 0$. First, a combination of (1) and (5) leads to $a_{11} = -a_{21} = -a_{33}$. If $a_{22} \neq 0$, a combination of (2) and (4) implies $a_{12} = 0$ and $a_{22} = a_{33}$ and the remaining equations are nullified. If $a_{22} = 0$, (4) states that $a_{12} = -a_{33}$ and the remaining equations are nullified. Consequently, we have two possible forms for $\phi$: 
    $\begin{brsm}-a & 0 & 0\\a & a & 0\\b & c & a\end{brsm}$, $\begin{brsm}-a & -a & 0\\a & 0 & 0\\b & c & a\end{brsm}$, for some $a \neq 0$.

    \item Suppose that $a_{21} = 0$. From invertibility, $a_{22} \neq 0$. If $a_{12} \neq 0$, a combination of (2) and (6) results in $a_{12} = -a_{22} = a_{33}$, (4) results in $a_{11} = a_{33}$, and the remaining equations are nullified. If $a_{12} = 0$, (2) and (4) imply $a_{11} = a_{22} = a_{33}$ and the remaining equations are nullified. Consequently, we have two possible forms for $\phi$: 
    $\begin{brsm}a & a & 0\\0 & -a & 0\\b & c & a\end{brsm}$, $\begin{brsm}a & 0 & 0\\0 & a & 0\\b & c & a\end{brsm}$, for some $a \neq 0$.
\end{itemize}

In conclusion,
\[
\text{PAut}_{\text{gr}}(P) = \left\lbrace
\begin{bmatrix}0 & a & 0\\-a & -a & 0\\b & c & a\end{bmatrix},
\begin{bmatrix}0 & -a & 0\\-a & 0 & 0\\b & c & a\end{bmatrix},
\begin{bmatrix}-a & 0 & 0\\a & a & 0\\b & c & a\end{bmatrix},
\color{white}\right\rbrace\\
\]
\[
\hspace{2.55cm}\color{white}\left\rbrace\color{black}
\begin{bmatrix}-a & -a & 0\\a & 0 & 0\\b & c & a\end{bmatrix},
\begin{bmatrix}a & a & 0\\0 & -a & 0\\b & c & a\end{bmatrix}, 
\begin{bmatrix}a & 0 & 0\\0 & a & 0\\b & c & a\end{bmatrix}
: a \neq 0\right\rbrace\color{black}.
\]

\smallskip

Again, let $\phi \in \text{PAut}_{\text{gr}}(P)$. As discussed above, $\phi$ has one of the following six forms: $\phi\big\vert_{P_1} = 
\begin{brsm}0 & a & 0\\-a & -a & 0\\b & c & a\end{brsm}$,\\
$\begin{brsm}0 & -a & 0\\-a & 0 & 0\\b & c & a\end{brsm},
\begin{brsm}-a & 0 & 0\\a & a & 0\\b & c & a\end{brsm},
\begin{brsm}-a & -a & 0\\a & 0 & 0\\b & c & a\end{brsm},
\begin{brsm}a & a & 0\\0 & -a & 0\\b & c & a\end{brsm}, 
\begin{brsm}a & 0 & 0\\0 & a & 0\\b & c & a\end{brsm}$. The eigenvalues of the first and fourth matrices are $\lambda_1 = a, \lambda_2, \lambda_3 = \displaystyle{\frac{(-1 \pm i\sqrt{3})a}{2}}$. If $\phi$ is a Poisson reflection, then $\{\lambda_1, \lambda_2, \lambda_3\} = \{1,1,\xi\}$ for some primitive root of unity $\xi$. In any case, $\lambda_2 + \lambda_3 \neq -\lambda_1$, a contradiction. The eigenvalues of the sixth matrix are $a, a, a$; consequently, it is impossible for such a matrix to be a Poisson reflection. The eigenvalues of the second, third, and fifth matrices are $\lambda_1 = -a, \lambda_2, \lambda_3 = a$. If $\phi$ is a Poisson reflection, then $a = 1$ and there exists three candidates for $\phi$:
$\begin{brsm}0 & -1 & 0\\-1 & 0 & 0\\b & c & 1\end{brsm}$, 
$\begin{brsm}-1 & 0 & 0\\1 & 1 & 0\\b & c & 1\end{brsm}$, 
$\begin{brsm}1 & 1 & 0\\0 & -1 & 0\\b & c & 1\end{brsm}$. By a straightforward calculation,

\[
\begin{bmatrix}0 & -1 & 0\\-1 & 0 & 0\\b & c & 1\end{bmatrix}^n = 
\begin{cases}
\begin{bmatrix}1 & 0 & 0\\0 & 1 & 0\\\frac{n(b-c)}{2} & \frac{n(c-b)}{2} & 1\end{bmatrix} & n \text{ is even,}\\
\vspace{-0.75cm}\\
\begin{bmatrix}0 & -1 & 0\\-1 & 0 & 0\\\frac{n+1}{2}b-\frac{n-1}{2}c & \frac{n+1}{2}c-\frac{n-1}{2}b & 1\end{bmatrix} & n \text{ is odd.}
\end{cases}
\]

In this case, a finite order Poisson reflection $\phi$ satisfies $b = c$.

\[
\begin{bmatrix}-1 & 0 & 0\\1 & 1 & 0\\b & c & 1\end{bmatrix}^n = 
\begin{cases}
\begin{bmatrix}1 & 0 & 0\\0 & 1 & 0\\\frac{n}{2}c & nc & 1\end{bmatrix} & n \text{ is even,}\\
\vspace{-0.75cm}\\
\begin{bmatrix}0 & -1 & 0\\-1 & 0 & 0\\b+\frac{n-1}{2}c & nc & 1\end{bmatrix} & n \text{ is odd.}
\end{cases}
\]

In this case, a finite order Poisson reflection $\phi$ satisfies $c = 0$.

\[
\begin{bmatrix}1 & 1 & 0\\0 & -1 & 0\\b & c & 1\end{bmatrix}^n = 
\begin{cases}
\begin{bmatrix}1 & 0 & 0\\0 & 1 & 0\\nb & \frac{n}{2}b & 1\end{bmatrix} & n \text{ is even,}\\
\vspace{-0.75cm}\\
\begin{bmatrix}1 & 1 & 0\\0 & -1 & 0\\nb & \frac{n-1}{2}b+c & 1\end{bmatrix} & n \text{ is odd.}
\end{cases}
\]

In this case, a finite order Poisson reflection $\phi$ satisfies $b = 0$.

In conclusion,
\[
\text{PR}(P) = \left\{
\begin{bmatrix}0 & -1 & 0\\-1 & 0 & 0\\b & b & 1\end{bmatrix},
\begin{bmatrix}-1 & 0 & 0\\1 & 1 & 0\\b & 0 & 1\end{bmatrix}, 
\begin{bmatrix}1 & 1 & 0\\0 & -1 & 0\\0 & c & 1\end{bmatrix}
\right\}.
\]
\end{proof}

\smallskip

\begin{mylem}
If $G$ is a finite Poisson reflection group of $P$, then the invariant subalgebra $P^G$ is not isomorphic to $P$ as Poisson algebras. 
\end{mylem}

\begin{proof}
As discussed above, a generator $\phi$ of $G$ has three possible forms: $\phi\big\vert_{P_1} = 
\begin{brsm}0 & -1 & 0\\-1 & 0 & 0\\b & b & 1\end{brsm}, 
\begin{brsm}-1 & 0 & 0\\1 & 1 & 0\\b & 0 & 1\end{brsm}$,
$\begin{brsm}1 & 1 & 0\\0 & -1 & 0\\0 & c & 1\end{brsm}$. First, notice that the finite automorphism group $G$ cannot contain two Poisson reflections of the same form.
\begin{itemize}
\item Suppose that $G$ contains two distinct Poisson reflections of the first form: $\phi_1\big\vert_{P_1} = \begin{brsm}0 & -1 & 0\\-1 & 0 & 0\\b_1 & b_1 & 1\end{brsm}$, $\phi_2\big\vert_{P_1} = \begin{brsm}0 & -1 & 0\\-1 & 0 & 0\\b_2 & b_2 & 1\end{brsm}$. The product $(\phi_1\phi_2)\big\vert_{P_1} = \begin{brsm}1 & 0 & 0\\0 & 1 & 0\\b_2-b_1 & b_2-b_1 & 1\end{brsm}$ is a matrix of infinite order, contradicting to the finiteness of $G$. 

\item Suppose that $G$ contains two distinct Poisson reflections of the second form: $\phi_1\big\vert_{P_1} = \begin{brsm}-1 & 0 & 0\\1 & 1 & 0\\b_1 & 0 & 1\end{brsm}$, $\phi_2\big\vert_{P_1} = \begin{brsm}-1 & 0 & 0\\1 & 1 & 0\\b_2 & 0 & 1\end{brsm}$. The product $(\phi_1\phi_2)\big\vert_{P_1} = \begin{brsm}1 & 0 & 0\\0 & 1 & 0\\b_2-b_1 & 0 & 1\end{brsm}$ is a matrix of infinite order, contradicting to the finiteness of $G$. 

\item Suppose that $G$ contains two distinct Poisson reflections of the third form: $\phi_1\big\vert_{P_1} = \begin{brsm}1 & 1 & 0\\0 & -1 & 0\\0 & c_1 & 1\end{brsm}$, $\phi_2\big\vert_{P_1} = \begin{brsm}1 & 1 & 0\\0 & -1 & 0\\0 & c_2 & 1\end{brsm}$. The product $(\phi_1\phi_2)\big\vert_{P_1} = \begin{brsm}1 & 0 & 0\\0 & 1 & 0\\0 & b_2-b_1 & 1\end{brsm}$ is a matrix of infinite order, contradicting to the finiteness of $G$. 
\end{itemize}

Consequently, the finite Poisson reflection group $G$ falls into one of the following three categories:
\begin{enumerate}[label = (\arabic*)]
    \item The group $G$ is generated by three Poisson reflections of distinct types.

    \item The group $G$ is generated by two Poisson reflections of distinct types.

    \item The group $G$ is generated by a single Poisson reflection.
\end{enumerate}

In Case (1), $G = \langle 
\phi_1\big\vert_{P_1} = \begin{brsm}0 & -1 & 0\\-1 & 0 & 0\\a & a & 1\end{brsm}, 
\phi_2\big\vert_{P_1} = \begin{brsm}-1 & 0 & 0\\1 & 1 & 0\\b & 0 & 1\end{brsm},
\phi_3\big\vert_{P_1} = \begin{brsm}1 & 1 & 0\\0 & -1 & 0\\0 & c & 1\end{brsm}
\rangle$. By calculation, the product $(\phi_1\phi_2\phi_3)\big\vert_{P_1}^n$ is not equal to the identity matrix $I_3$ when $n$ is odd, and it equals the following matrix:
\[
(\phi_1\phi_2\phi_3)\big\vert_{P_1}^n = 
\begin{bmatrix}1 & 0 & 0\\0 & 1 & 0\\\frac{n}{2}(b+c-a) & n(b+c-a) & 1\end{bmatrix}
\]
when $n$ is even. As a result, the finiteness of $G$ necessitates the condition: $a = b + c$. Given this equality, through further calculations, we observe that:
\begin{alignat*}{3}
(\phi_1\phi_2\phi_1)\big\vert_{P_1} =& 
\begin{bmatrix}0 & -1 & 0\\-1 & 0 & 0\\b+c & b+c & 1\end{bmatrix}
\begin{bmatrix}-1 & 0 & 0\\1 & 1 & 0\\b & 0 & 1\end{bmatrix}
\begin{bmatrix}0 & -1 & 0\\-1 & 0 & 0\\b+c & b+c & 1\end{bmatrix} =& 
\begin{bmatrix}1 & 1 & 0\\0 & -1 & 0\\0 & c & 1\end{bmatrix}
=& \phi_3\big\vert_{P_1},\\
(\phi_2\phi_3\phi_2)\big\vert_{P_1} =& 
\begin{bmatrix}-1 & 0 & 0\\1 & 1 & 0\\b & 0 & 1\end{bmatrix}
\begin{bmatrix}1 & 1 & 0\\0 & -1 & 0\\0 & c & 1\end{bmatrix}
\begin{bmatrix}-1 & 0 & 0\\1 & 1 & 0\\b & 0 & 1\end{bmatrix} =&
\begin{bmatrix}0 & -1 & 0\\-1 & 0 & 0\\b+c & b+c & 1\end{bmatrix}
=& \phi_1\big\vert_{P_1},\\
(\phi_3\phi_1\phi_3)\big\vert_{P_1} =& 
\begin{bmatrix}1 & 1 & 0\\0 & -1 & 0\\0 & c & 1\end{bmatrix}
\begin{bmatrix}0 & -1 & 0\\-1 & 0 & 0\\b+c & b+c & 1\end{bmatrix}
\begin{bmatrix}1 & 1 & 0\\0 & -1 & 0\\0 & c & 1\end{bmatrix} =& 
\begin{bmatrix}-1 & 0 & 0\\1 & 1 & 0\\b & 0 & 1\end{bmatrix}
=& \phi_2\big\vert_{P_1}.
\end{alignat*}

In conclusion, we have established that $G = \langle\phi_1, \phi_2, \phi_3\rangle = \langle \phi_1, \phi_2 \rangle = \langle \phi_2, \phi_3 \rangle = \langle \phi_3, \phi_1 \rangle$, and it is sufficient to consider the case when $G = \langle \phi_1, \phi_2 \rangle$ for both Case (1) and Case (2).

\smallskip

In Case (2), $G = \langle 
\phi_1\big\vert_{P_1} = \begin{brsm}0 & -1 & 0\\-1 & 0 & 0\\a & a & 1\end{brsm}, 
\phi_2\big\vert_{P_1} = \begin{brsm}-1 & 0 & 0\\1 & 1 & 0\\b & 0 & 1\end{brsm}
\rangle$. By calculation,
\[
\phi_1^2\big\vert_{P_1} = \phi_2^2\big\vert_{P_1} = (\phi_1\phi_2)^3\big\vert_{P_1} = I_3, \quad \phi_1\phi_2 \neq \phi_2\phi_1.
\]
The first equality implies that $G$ is isomorphic to a quotient of the symmetric group $S_3$. Further, the second inequality implies that $G$ is isomorphic to $S_3$, as $S_3$ is the smallest finite non-abelian group. Apply the Molien's Theorem:
\[
h_{P^G}(t) = \frac{1}{6}\left(\frac{1}{(1-t)^3} + \frac{3}{(1-t)^2(1+t)} + \frac{2}{(1-t)(1+t+t^2)}\right)\\
= \frac{1}{(1-t)(1-t^2)(1-t^3)}.
\] 

In the meantime, notice that the elements $\displaystyle{y_1 = \frac{1}{3}(a+b)x_1+\frac{1}{3}(2a-b)x_2+x_3}$, $y_2 = x_1^2+x_2^2$ $+x_1x_2$, $\displaystyle{y_3 = 2x_1^3+3x_1^2x_2-3x_1x_2^2-2x_2^3}$ are three algebraically independent polynomials that are invariant under the action of $G$. Embed $\Bbbk[y_1,y_2,y_3]$ into the invariant subalgebra $P^G$. Since the domain and codomain share the same Hilbert series, the cokernel is trivial. In other words, the invariant subalgebra $P^G = \Bbbk[y_1,y_2,y_3]$ and has the following Poisson structure:
\[
\{y_1,y_2\} = y_3, \quad \{y_2, y_3\} = 0, \quad \{y_3, y_1\} = -6y_2^2.
\]
By applying Lemma 2.2, $P^G$ is not isomorphic to $P$ as Poisson algebras. 

\smallskip

In Case (3), we will concurrently discuss three subcases: $G_1 = \langle\phi_1\big\vert_{P_1} = \begin{brsm}0 & -1 & 0\\-1 & 0 & 0\\a & a & 1\end{brsm}\rangle$, 
$G_2 = \langle \phi_2\big\vert_{P_1} = \begin{brsm}-1 & 0 & 0\\1 & 1 & 0\\b & 0 & 1\end{brsm}\rangle$, 
and $G_3 = \langle\phi_3\big\vert_{P_1} = \begin{brsm}1 & 1 & 0\\0 & -1 & 0\\0 & c & 1\end{brsm}\rangle$. Notice that the finite Poisson reflection group $G_i$ is isomorphic to $\mathbb{Z}_2$ for all $1 \leq i \leq 3$, and the Hilbert series of the invariant subalgebra is
\[
h_{P^G}(t) = \frac{1}{2}\left(\frac{1}{(1-t)^3} + \frac{1}{(1-t)^2(1+t)}\right) = \frac{1}{(1-t)^2(1-t^2)}
\]
by applying Molien's Theorem. As in our previous analysis, we can find three algebraically independent polynomials that are invariant under the action of $G_i$:
\begin{center}
\renewcommand{\arraystretch}{1.25}
\begin{tabular}{ |p{1.5cm}||p{4cm}|p{4cm}|p{4cm}|} 
\hline
\centering &
\centering $\boldsymbol{y_1}$ &
\centering $\boldsymbol{y_2}$ &
\centering $\boldsymbol{y_3}$ \tabularnewline
\hline
\hline
\centering $G_1$ &
\centering $x_1x_2$ & \centering $-x_1+x_2$ & \centering $ax_1+x_3$
\tabularnewline
\hline
\centering $G_2$ &
\centering $x_1^2$ & \centering $\frac{b}{2}x_1+x_3$ & \centering $\frac{1}{2}x_1+x_2$
\tabularnewline
\hline
\centering $G_3$ &
\centering $x_2^2$ & \centering $-cx_1+x_3$ & \centering $2x_1+x_2$
\tabularnewline
\hline
\end{tabular}
\end{center}

By comparing the Hilbert series of the domain and the codomain of the embedding $\Bbbk[y_1, y_2, y_3] \hookrightarrow P^G$, we ascertain that the invariant subalgebra $P^G = \Bbbk[y_1, y_2, y_3]$ and has the following Poisson structures:
\begin{center}
\renewcommand{\arraystretch}{1.25}
\begin{tabular}{ |p{1.5cm}||p{4cm}|p{4cm}|p{4cm}|} 
\hline
\centering &
\centering $\boldsymbol{\{y_1, y_2\}}$ &
\centering $\boldsymbol{\{y_2, y_3\}}$ &
\centering $\boldsymbol{\{y_3, y_1\}}$ \tabularnewline
\hline
\hline
\centering $G_1$ &
\centering $0$ & \centering $6y_1+y_2^2$ & \centering $y_1y_2$
\tabularnewline
\hline
\centering $G_2$ &
\centering $-4y_1y_3$ & \centering $\frac{3}{4}y_1-y_3^2$ & \centering $0$
\tabularnewline
\hline
\centering $G_3$ &
\centering $2y_1y_3$ & \centering $-\frac{3}{2}y_1 + \frac{1}{2}y_3^2$ & \centering $0$
\tabularnewline
\hline
\end{tabular}
\end{center}

Suppose that $P^G \cong P$ as Poisson algebras. Specifically, by applying (the contrapositive of) Lemma 2.1, the invariant subalgebra is unimodular. This implies that we can find a superpotential $\Omega \in \Bbbk[y_1, y_2, y_3]$ satisfying:
\[
\frac{\partial \Omega}{\partial y_3} = \{y_1, y_2\}, \quad \frac{\partial \Omega}{\partial y_1} = \{y_2, y_3\}, \quad \frac{\partial \Omega}{\partial y_2} = \{y_3, y_1\}.
\]
It is straightforward to verify no such $\Omega$ exists for $G_1$, $G_2$, and $G_3$, and therefore, $P^G$ is not unimodular, a contradiction. 
\end{proof}

\smallskip
\subsection{\texorpdfstring{$\Omega_5 = x_1^3 + x_1^2x_3, \quad \{x_1,x_2\} = x_2^2, \quad \{x_2,x_3\} = 3x_1^2, \quad \{x_3,x_1\} = 2x_2x_3$}{Case 5}}

\begin{mylem}
\hspace{.1cm}
\begin{center}
\renewcommand{\arraystretch}{1.25}
\begin{tabular}{ |p{8cm}|p{8cm}|} 
\hline
\centering ${\textbf{PAut}_{{\textbf{gr}}}\boldsymbol{(P)}}$ &
\centering $\textbf{PR(}\boldsymbol{P}\textbf{)}$ \tabularnewline
\hline
\hline
\centering $\left\lbrace
\begin{bmatrix}
a & 0 & 0\\
0 & a & 0\\
0 & 0 & a
\end{bmatrix}: a \neq 0\right\rbrace$ &
\centering $\emptyset$
\tabularnewline
\hline
\end{tabular}
\end{center}
\end{mylem}

\begin{proof}
Let $\phi \in \text{PAut}_{\text{gr}}(P)$. Apply (2.1), we have the following system of equations, with redundant equations omitted:
\begin{enumerate}[label = (\arabic*)]

    \item $a_{21}^2 = 3a_{12}a_{23}-3a_{13}a_{22}$.
    
    \item $a_{23}^2 = 0$.
    
    \item $a_{11}^2 = a_{22}a_{33}-a_{23}a_{32}$.
        
    \item $a_{13}^2 = 0$.
    
    \item $a_{11}a_{12} = 0$. 
    
    \item $2a_{22}a_{32} = a_{12}a_{31}-a_{11}a_{32}$.
        
    \item $a_{21}a_{32}+a_{22}a_{31} = 0$.
        
    \item $a_{22}a_{33}+a_{23}a_{32} = a_{11}a_{33}-a_{13}a_{31}$.
\end{enumerate}

It is immediate from (2) and (4) that $a_{13} = a_{23} = 0$ and $a_{33} \neq 0$. Expanding upon these, (1) and (8) suggest that $a_{21} = 0$ and $a_{11} = a_{22} \neq 0$. Continuing further, (3) and (5) imply that $a_{11} = a_{33}$ and $a_{12} = 0$. Finally, (6) and (7) necessitate that $a_{31} = a_{32} = 0$. In conclusion,
\[
\text{PAut}_{\text{gr}}(P) = \left\lbrace
\begin{bmatrix}
a & 0 & 0\\
0 & a & 0\\
0 & 0 & a
\end{bmatrix}: a \neq 0\right\rbrace.
\]

It is clear that there are no Poisson reflections for this Poisson structure, as any graded Poisson automorphism has three repeated eigenvalues.
\end{proof}

\smallskip
\subsection{\texorpdfstring{$\Omega_6 = x_1^3 + x_1^2x_3 + x_2x_3, \quad \{x_1, x_2\} = x_1^2+x_2^2, \quad \{x_2, x_3\} = 3x_1^2+2x_1x_3, \quad \{x_3, x_1\} = 2x_2x_3$}{Case 6}}

\begin{mylem}
\hspace{.1cm}
\begin{center}
\renewcommand{\arraystretch}{1.25}
\begin{tabular}{ |p{8cm}|p{8cm}|} 
\hline
\centering ${\textbf{PAut}_{{\textbf{gr}}}\boldsymbol{(P)}}$ &
\centering $\textbf{PR(}\boldsymbol{P}\textbf{)}$ \tabularnewline
\hline
\hline
\centering $\left\lbrace
\begin{bmatrix}a & 0 & 0\\0 & a & 0\\0 & 0 & a\end{bmatrix}, 
\begin{bmatrix}-\frac{1}{2}a & \pm\frac{\sqrt{3}}{2}a & 0\\\mp\frac{\sqrt{3}}{2}a & -\frac{1}{2}a & 0\\\frac{9}{8}a & \mp\frac{3\sqrt{3}}{8}a & a\end{bmatrix}
: a \neq 0\right\rbrace$ &
\centering $\emptyset$
\tabularnewline
\hline
\end{tabular}
\end{center}
\end{mylem}

\begin{proof}
Let $\phi \in \text{PAut}_{\text{gr}}(P)$. Apply (2.1), we have the following system of equations, with redundant equations omitted:
\begin{enumerate}[label = (\arabic*)]

    \item $a_{13}^2 + a_{23}^2 = 0.$
    
    \item $3a_{11}^2+2a_{11}a_{31} = a_{21}a_{32}-a_{22}a_{31}+3a_{22}a_{33}-3a_{23}a_{32}$.

    \item $3a_{12}^2+2a_{12}a_{32} = a_{21}a_{32}-a_{22}a_{31}$.

    \item $3a_{13}^2+2a_{13}a_{33} = 0.$
    
    \item $3a_{11}a_{12}+a_{11}a_{32}+a_{12}a_{31} = 0$.
    
    \item $3a_{11}a_{13}+a_{11}a_{33}+a_{13}a_{31} = a_{22}a_{33}-a_{23}a_{32}$.

    \item $3a_{12}a_{13}+a_{12}a_{33}+a_{13}a_{32} = a_{23}a_{31}-a_{21}a_{33}$.

    \item $2a_{21}a_{31} = a_{12}a_{31}-a_{11}a_{32} + 3a_{13}a_{32}-3a_{12}a_{33}$.
        
    \item $a_{23}a_{33} = 0.$

    \item $a_{21}a_{32}+a_{22}a_{31} = 0.$
\end{enumerate}

Suppose that $a_{33} = 0$. Equation (1) and (4) necessitate that $a_{13} = a_{23} = 0$, a contradiction to the invertibility. Therefore, it follows that $a_{33} \neq 0$, and subsequently, according to equations (1) and (9), $a_{13} = a_{23} = 0$. From (6) and (7), $a_{11} = a_{22}$ and $a_{12} = -a_{21}$. 

Let us assume that $a_{12} = 0$ (implying $a_{21} = 0$ implicitly). From (5), (10), and the invertibility, $a_{31} = a_{32} = 0$. Finally, from (2), $a_{11} = a_{22} = a_{33}$ and the remaining equations are nullified. This results in one possible form of $\phi$: 
$\begin{brsm}a & 0 & 0\\0 & a & 0\\0 & 0 & a\end{brsm}$ for some $a \neq 0$. Now, let us consider the alternative scenario $a_{12} \neq 0$. From a combination of (3) and (10), we can derive that $a_{32} = -\frac{3}{4}a_{12}$. Substituting our results into (5), we obtain that $a_{31} = -\frac{9}{4}a_{11}$. Examining (8), we observe that $a_{11} = -\frac{1}{2}a_{33}$. Lastly, from (10), $a_{12} = \pm\sqrt{3}a_{11}$ and the remaining equations are nullified. This results in one possible form of $\phi$: 
$\begin{brsm}-\frac{1}{2}a & \pm\frac{\sqrt{3}}{2}a & 0\\\mp\frac{\sqrt{3}}{2}a & -\frac{1}{2}a & 0\\\frac{9}{8}a & \mp\frac{3\sqrt{3}}{8}a & a\end{brsm}$ for some $a \neq 0$. In conclusion,
\[
\text{PAut}_{\text{gr}}(P) = \left\lbrace
\begin{bmatrix}a & 0 & 0\\0 & a & 0\\0 & 0 & a\end{bmatrix}, 
\begin{bmatrix}-\frac{1}{2}a & \pm\frac{\sqrt{3}}{2}a & 0\\\mp\frac{\sqrt{3}}{2}a & -\frac{1}{2}a & 0\\\frac{9}{8}a & \mp\frac{3\sqrt{3}}{8}a & a\end{bmatrix}
: a \neq 0\right\rbrace.
\]

\smallskip

Again, let $\phi \in \text{PAut}_{\text{gr}}(P)$. If $\phi\big\vert_{P_1} = 
\begin{brsm}a & 0 & 0\\0 & a & 0\\0 & 0 & a\end{brsm}$, then $\phi$ is not a Poisson reflection because it has three repeated eigenvalues. If $\phi\big\vert_{P_1} = \begin{brsm}-\frac{1}{2}a & \pm\frac{\sqrt{3}}{2}a & 0\\\mp\frac{\sqrt{3}}{2}a & -\frac{1}{2}a & 0\\\frac{9}{8}a & \mp\frac{3\sqrt{3}}{8}a & a\end{brsm}$, its eigenvalues are $\lambda_1 = a, \lambda_2, \lambda_3 = \displaystyle{\frac{(-1 \pm i\sqrt{3})a}{2}}$ and therefore it is not a Poisson reflections (as discussed in subsection 3.3).
\end{proof} 

\smallskip
\subsection{\texorpdfstring{$\Omega_7 = \frac{1}{3}(x_1^3 + x_2^3 + x_3^3) + \lambda x_1x_2x_3 \hspace{.1cm}(\lambda^3 \neq -1),\\ \color{white}aaaa\color{black} \{x_1, x_2\} = x_3^2+\lambda x_1x_2, \quad \{x_2, x_3\} = x_1^2 + \lambda x_2x_3, \quad \{x_3, x_1\} = x_2^2 + \lambda x_1x_3$}{Case 7}}

\clearpage

\begin{mylem}
\hspace{.1cm}
\begin{center}
\renewcommand{\arraystretch}{1.25}
\begin{tabular}{ |p{12cm}|p{4cm}|} 
\hline
\centering ${\textbf{PAut}_{{\textbf{gr}}}\boldsymbol{(P)}}$ &
\centering $\textbf{PR(}\boldsymbol{P}\textbf{)}$ \tabularnewline
\hline
\hline
\centering $\left(
\langle\begin{bmatrix}a & 0 & 0\\0 & a & 0\\0 & 0 & a\end{bmatrix}: a \neq 0\rangle \times \langle\begin{bmatrix}1 & 0 & 0\\0 & b & 0\\0 & 0 & b^2\end{bmatrix}: b = \xi_3, \xi_3^2\rangle
\right) \rtimes 
\langle\begin{bmatrix}0 & 1 & 0\\0 & 0 & 1\\1 & 0 & 0\end{bmatrix}\rangle$ &
\centering $\emptyset$
\tabularnewline
\hline
\end{tabular}
\end{center}
\end{mylem}

\begin{proof}
The Poisson automorphism group $\text{PAut}_{\text{gr}}(P)$ has been classified in \cite[Theorem 1]{MLTU}. For Poisson reflections, Gaddis, Veerapen, and Wang have provided a comprehensive analysis of the linear Poisson normal elements pertaining to this particular instance \cite[Lemma 4.3]{GVW}. Given that $\lambda^3 \neq 1$, it follows that there are no linear Poisson normal elements, thereby precluding the existence of Poisson reflections \cite[Lemma 2.2]{GVW}. 
\end{proof}

\smallskip
\subsection{\texorpdfstring{$\Omega_8 = x_1^3 + x_1^2x_2 + x_1x_2x_3, \quad \{x_1, x_2\} = x_1x_2, \quad \{x_2, x_3\} = 3x_1^2+2x_1x_2+x_2x_3, \quad \{x_3, x_1\} = x_1^2+x_1x_3$}{Case 8}}

\begin{mylem}
\hspace{.1cm}
\begin{center}
\renewcommand{\arraystretch}{1.25}
\begin{tabular}{ |p{8cm}|p{8cm}|} 
\hline
\centering ${\textbf{PAut}_{{\textbf{gr}}}\boldsymbol{(P)}}$ &
\centering $\textbf{PR(}\boldsymbol{P}\textbf{)}$ \tabularnewline
\hline
\hline
\centering $\left\lbrace
\begin{bmatrix}a & 0 & 0\\0 & \frac{a^2}{b} & 0\\b-a & 0 & b\end{bmatrix}: a \neq 0\right\rbrace$ &
\centering $\left\{
\begin{bmatrix}-1 & 0 & 0\\0 & 1 & 0\\2 & 0 & 1\end{bmatrix}
\right\}$
\tabularnewline
\hline
\end{tabular}
\end{center}
\end{mylem}

\begin{proof}
Let $\phi \in \text{PAut}_{\text{gr}}(P)$. Apply (2.1), we have the following system of equations, with redundant equations omitted:
\begin{enumerate}[label = (\arabic*)]
    \item $a_{11}a_{21} = 3a_{12}a_{23}-3a_{13}a_{22}+a_{13}a_{21}-a_{11}a_{23}$.
    
    \item $a_{12}a_{22} = 0$.

    \item $a_{13}a_{23} = 0$.

    \item $a_{11}a_{23}+a_{13}a_{21} = a_{13}a_{21}-a_{11}a_{23}$.
        
    \item $3a_{11}^2+a_{21}a_{31}+2a_{11}a_{21} = 3a_{22}a_{33}-3a_{23}a_{32}+a_{23}a_{31}-a_{21}a_{33}$.
    
    \item $3a_{12}^2+a_{22}a_{32}+2a_{12}a_{22} = 0$.

    \item $3a_{13}^2+a_{23}a_{33}+2a_{13}a_{23} = 0$.

    \item $6a_{11}a_{12}+a_{21}a_{32}+a_{22}a_{31}+2a_{11}a_{22}+2a_{12}a_{21} = a_{21}a_{32}-a_{22}a_{31}+2a_{22}a_{33}-2a_{23}a_{32}$.

    \item $a_{11}a_{32}+a_{12}a_{31}+2a_{11}a_{12} = a_{12}a_{31}-a_{11}a_{32}+2a_{13}a_{32}-2a_{12}a_{33}.$
\end{enumerate}

Suppose that $a_{12} \neq 0$. By combining (2) and (6), we can deduce that $a_{12} = 0$, a contradiction. Consequently, it follows that $a_{12} = 0$. From (3), (4), and the invertibility, $a_{23} = 0$. Immediately, (7) translates to $a_{13} = 0$ and $a_{11}, a_{33} \neq 0$. Subsequently, it can be inferred from (1) and (9) that $a_{21} = a_{32} = 0$ and $a_{22} \neq 0$. Based on (5) and (8), it follows that $a_{22} = \frac{a_{11}^2}{a_{33}}$ and $a_{31} = a_{33}-a_{11}$, respectively. In conclusion, 
\[
\text{PAut}_{\text{gr}}(P) = \left\lbrace
\begin{bmatrix}a & 0 & 0\\0 & \frac{a^2}{b} & 0\\b-a & 0 & b\end{bmatrix}: a \neq 0\right\rbrace.
\]

\smallskip

Again, let $\phi \in \text{PAut}_{\text{gr}}(P)$. Its eigenvalues are $\lambda_1 = a, \lambda_2 = \displaystyle{\frac{a^2}{b}}, \lambda_3 = b$. If $\phi$ is a Poisson reflection, then $\{\lambda_1, \lambda_2, \lambda_3\} = \{1, 1, \xi\}$ for some primitive root of unity $\xi$. If $\lambda_1 = 1$, then $\lambda_2\lambda_3 = \xi \neq \lambda_1^2$, a contradiction. If $\lambda_1 = \xi$, then the constraint $\lambda_2\lambda_3 = \lambda_1^2$ implies that $\xi = -1$ and $\phi\big\vert_{P_1}$ takes the form $\begin{brsm}-1 & 0 & 0\\0 & 1 & 0\\2 & 0 & 1\end{brsm}$. Such $\phi$ has finite order 2. In conclusion,
\[
\text{PR}(P) = \left\{
\begin{bmatrix}-1 & 0 & 0\\0 & 1 & 0\\2 & 0 & 1\end{bmatrix}
\right\}.
\]
\end{proof}

\smallskip

\begin{mylem}
If $G$ is a finite Poisson reflection group of $P$, then the invariant subalgebra $P^G$ is not isomorphic to $P$ as Poisson algebras. 
\end{mylem}

\begin{proof}
By Lemma 3.8.1, the finite Poisson reflection group $G = \langle \phi: \phi\big\vert_{P_1} = \begin{brsm}-1 & 0 & 0\\0 & 1 & 0\\2 & 0 & 1\end{brsm} \rangle \cong \mathbb{Z}_2$. By the Molien's Theorem, 
\[
h_{P^G}(t) = \frac{1}{2}\left(\frac{1}{(1-t)^3} + \frac{1}{(1-t)^2(1+t)}\right) = \frac{1}{(1-t)^2(1-t^2)}.
\]

It's not difficult to find three elements $y_1 = x_2$, $y_2 = x_1+x_3$, $y_3 = x_1^2$ that are algebraically independent and are invariant under the action of $G$. Consider the natural inclusion $\Bbbk[y_1, y_2, y_3] \hookrightarrow P^G$. Extend the natural inclusion into a short exact sequence and apply Hilbert series. Given that the Hilbert series of a short exact sequence sums up to 0, we deduce that the invariant subalgebra $P^G = \Bbbk[y_1, y_2, y_3]$. The Poisson structure on the invariant subalgebra is:
\[
\{y_1, y_2\} = y_1y_2 + 3y_3, \quad
\{y_2, y_3\} = 2y_2y_3, \quad
\{y_3, y_1\} = 2y_1y_3.
\]
To prove that $P^G$ is not isomorphic to $P$ as Poisson algebras, one can either invoke Lemma 2.1 or Lemma 2.2. Both approaches are equally straightforward.
\end{proof}

\smallskip
\subsection{\texorpdfstring{$\Omega_9 = x_1^2x_3 + x_1^2x_2, \quad \{x_1, x_2\} = x_1^2, \quad \{x_2, x_3\} = 2x_1x_3+x_2^2, \quad \{x_3, x_1\} = 2x_1x_2$}{Case 9}}

\begin{mylem}
\hspace{.1cm}
\begin{center}
\renewcommand{\arraystretch}{1.25}
\begin{tabular}{ |p{8cm}|p{8cm}|} 
\hline
\centering ${\textbf{PAut}_{{\textbf{gr}}}\boldsymbol{(P)}}$ &
\centering $\textbf{PR(}\boldsymbol{P}\textbf{)}$ \tabularnewline
\hline
\hline
\centering $\left\lbrace
\begin{bmatrix}a & 0 & 0\\b & a & 0\\-\frac{b^2}{a} & -2b & a\end{bmatrix}: a \neq 0\right\rbrace$ &
\centering $\emptyset$
\tabularnewline
\hline
\end{tabular}
\end{center}
\end{mylem}

\begin{proof}
Let $\phi \in \text{PAut}_{\text{gr}}(P)$. Apply (2.1), we have the following system of equations, with redundant equations omitted:
\begin{enumerate}[label = (\arabic*)]
    \item $a_{11}^2 = a_{11}a_{22}-a_{12}a_{21}$.
    
    \item $a_{12}^2 = a_{12}a_{23}-a_{13}a_{22}$.

    \item $a_{13}^2 = 0$.
    
    \item $a_{21}^2+2a_{11}a_{31} = a_{21}a_{32}-a_{22}a_{31}$. 
    
    \item $a_{22}^2+2a_{12}a_{32} = a_{22}a_{33}-a_{23}a_{32}$.

    \item $a_{23}^2+2a_{13}a_{33} = 0$. 
    
    \item $2a_{11}a_{21} = a_{12}a_{31}-a_{11}a_{32}.$
\end{enumerate}

The initial step is straightforward. Equations (3), (6), (2) imply $a_{12} = a_{13} = a_{23} = 0$ and $a_{11}, a_{33} \neq 0$. Simplify the remaining equations. From (1) and (5), we conclude that $a_{11} = a_{22} = a_{33}$. From (7), we deduce that $a_{32} = -2a_{21}$. Finally, by examining (4), we ascertain that $a_{31} = -\frac{a_{21}^2}{a_{11}}$. In conclusion,
\[
\text{PAut}_{\text{gr}}(P) = \left\lbrace
\begin{bmatrix}a & 0 & 0\\b & a & 0\\-\frac{b^2}{a} & -2b & a\end{bmatrix}: a \neq 0\right\rbrace.
\]

It is clear that there are no Poisson reflections for this Poisson structure, as any graded Poisson automorphism has three repeated eigenvalues.
\end{proof}

\smallskip

We are now prepared to prove Theorem 3.1.

\smallskip

\noindent\textit{Proof of Theorem 3.1.} It suffices to prove $\Rightarrow$. The classical Shephard-Todd-Chevalley Theorem states that if $P^G \cong P$ as algebras, then $G$ is generated by reflections. Since $G \subseteq \text{PAut}_{\text{gr}}(P)$, the group $G$ is generated by Poisson reflections. By Lemma 3.$i$.1 ($1 \leq i \leq 9$), \textit{Unimodular 5}, \textit{Unimodular 6}, \textit{Unimodular 7}, \textit{Unimodular 9} have no Poisson reflections, hence the statement is trivially true. By Lemma 3.$i$.2 ($1 \leq i \leq 9$), the group $G$ is necessarily trivial for \textit{Unimodular 1}, \textit{Unimodular 2}, \textit{Unimodular 3}, \textit{Unimodular 4}, \textit{Unimodular 8}. \qed

\

\section{A Variant of The Watanabe Theorem}

In this section, we prove a variant of the Shephard-Todd-Chevalley Theorem and a variant of the Watanabe Theorem for Poisson enveloping algebras of quadratic Poisson structures on $\Bbbk[x_1, \cdots, x_n]$, under the actions induced by $\text{PAut}_{\text{gr}}(P) \to \text{Aut}_{\text{gr}}(U(P))$. 

\smallskip

Let $P = \Bbbk[x_1,\cdots,x_n]$ be a quadratic Poisson algebra and let $U(P)$ be its Poisson enveloping algebra. Let $\phi \in \text{PAut}_{\text{gr}}(P)$. Recall that in section 2, we define $\widetilde{\phi}$ to be the following graded algebra automorphism of $U(P)$:
\begin{align*}
    \widetilde{\phi}(x_i) = \phi(x_i), \quad \widetilde{\phi}(y_i) = \sum_{j=1}^{n}\frac{\partial g(x_i)}{\partial x_j}y_j,
\end{align*}
for all $1 \leq i \leq n$. Lemma 2.3 and Lemma 2.4 allow us to formulate the following questions, as generalizations of the Shephard-Todd-Chevalley Theorem and the Watanabe Theorem:
\begin{enumerate}[label = (\arabic*)]
    \item Under what conditions on $G$ is $U(P)^{\widetilde{G}}$ Artin-Schelter regular?
    
    \item Under what conditions on $G$ is $U(P)^{\widetilde{G}}$ Artin-Schelter Gorenstein?
\end{enumerate}

The answer to Question (1) entails establishing the absence of quasi-reflections in $\widetilde{G}$ through the enumeration of the eigenvalues of its elements, given that the set of eigenvalues of a quasi-reflection of $U(P)$ exhibits a highly constrained form as elucidated in \cite[Theorem 3.1]{KKZ}. 

\begin{thm}
Let $P = \Bbbk[x_1, \cdots, x_n]$ be a Poisson algebra and $U(P)$ be its Poisson enveloping algebra. Let $G$ be a finite nontrivial subgroup of the graded Poisson automorphism group of $P$ and let $\widetilde{G}$ be the corresponding finite nontrivial subgroup of the graded automorphism group of $U(P)$. The invariant subalgebra $U(P)^{\widetilde{G}}$ is Artin-Schelter regular if and only if $G$ is trivial.
\end{thm}
\begin{proof}
It suffices to assume the graded Poisson automorphism group $G$ is nontrivial. The Poisson enveloping algebra $U(P)$, according to \cite[Lemma 5.4]{GVW}, is a quantum polynomial ring. Therefore its quasi-reflections, as discussed in \cite[Theorem 3.1]{KKZ}, is either a classical reflection or a mystic reflection:
 \begin{itemize}
    \item The eigenvalues of $\widetilde{\phi}$ are $\underbrace{1,\cdots,1}_{2n-1}, \xi$, for some primitive root of unity $\xi$.
    
    \item The order of $\widetilde{\phi}$ is 4 and the eigenvalues of $\widetilde{\phi}$ are $\underbrace{1,\cdots,1}_{2n-2}$, $i$, $-i$. 
\end{itemize}
Comparing to the eigenvalues enumerated in Lemma 2.5, the induced graded automorphism group $\widetilde{G}$ contains no quasi-reflections, which, as indicated in \cite[Lemma 6.1]{KKZ}, results in the invariant subalgebra $U(P)^{\widetilde{G}}$ having infinite global dimension. As a consequence, the invariant subalgebra $U(P)^{\widetilde{G}}$ is not Artin-Schelter regular. 
\end{proof}

\smallskip

In contrast to Theorem 4.1, $U(P)^{\widetilde{G}}$ may be Artin–Schelter Gorenstein in certain instances. In practice, it is challenging to verify this through homological approaches because, in general, there is no systematic way to describe the relations of $U(P)^{\widetilde{G}}$. Instead, we shift our attention to \cite[Theorem 3.3]{JZ}:
\begin{thm}
\cite[Theorem 3.3]{JZ} Let $A$ be a Noetherian Artin-Schelter Gorenstein $\Bbbk$-algebra and let $H$ be a finite subgroup of the graded automorphism group of $A$. If $\text{hdet} h = 1$ for all $h \in H$, then $A^H$ is Artin-Schelter Gorenstein.
\end{thm}

To apply \cite[Theorem 3.3]{JZ}, we need a formula of computing the homological determinant of each induced graded automorphism $\widetilde{\phi}$ of $U(P)$, which will be the primary goal of the following lemma. 

\begin{lem}
Let $P = \Bbbk[x_1, \cdots, x_n]$ be a quadratic Poisson algebra and let $\phi$ be a finite-order graded Poisson automorphism of $P$. Suppose that $\phi\big\vert_{P_1}$ has eigenvalues $\lambda_1, \cdots, \lambda_m$, with multiplicity $c_1, \cdots, c_m$, respectively. Then
\[
\text{Tr}_{U(P)}(\widetilde{\phi},t) = \frac{1}{(1-\lambda_1t)^{2c_1} \cdots (1-\lambda_mt)^{2c_m}}.
\]
\end{lem}
\begin{proof}
Let $1 + \displaystyle{\sum_{i \geq 1}a_it^i}$ be the Taylor expansion of $\text{Tr}_{P}(\phi,t)$, where $a_i \in \Bbbk$ for all $i \geq 1$. It is sufficient to prove that $(1 + \displaystyle{\sum_{i \geq 1}a_it^i})^2 = \text{Tr}_{U(P)}(\widetilde{\phi},t)$. Fix $d \geq 1$. By \cite[Theorem 3.7]{OPS}, the degree $d$ component of the Poisson enveloping algebra $U(P)_d$ admits a $\Bbbk$-linear basis $\{x_1^{p_1}\cdots x_n^{p_n}y_1^{q_1} \cdots y_n^{q_n}: \displaystyle{\sum_{j=1}^{n} (p_j+q_j)} = d\}$. Let $r_1, \cdots, r_n \geq 0$, and $b_{r_1,\cdots,r_n}$ be the coefficient of the term $x_1^{r_1} \cdots x_n^{r_n}$ in $\phi(x_1^{r_1} \cdots x_n^{r_n})$. Consider the coefficient of the term $x_1^{p_1} \cdots x_n^{p_n}y_1^{q_1} \cdots y_n^{q_n}$ in $\widetilde{\phi}(x_1^{p_1} \cdots x_n^{p_n}y_1^{q_1} \cdots y_n^{q_n})$. There are three observations:
\begin{enumerate}[label = (\arabic*)]
    \item Given that $[x_i, x_j] = 0$ in $U(P)$, the coefficient of the term $x_1^{p_1} \cdots x_n^{p_n}$ in $\widetilde{\phi}(x_1^{p_1} \cdots x_n^{p_n})$ is $b_{p_1, \cdots, p_n}$. 

    \item Given that $[y_i,y_j] = \displaystyle{\sum_{k=1}^{n}\frac{\partial \{x_i,x_j\}}{\partial x_k}y_k}$ and $[y_i,x_j] = \{x_i,x_j\}$ in $U(P)$, the coefficient of the term $y_1^{q_1} \cdots y_n^{q_n}$ in $\widetilde{\phi}(y_1^{q_1} \cdots y_n^{q_n})$ is $b_{q_1, \cdots, q_n}$. 

    \item Given that $\widetilde{\phi}(x_j) \subseteq \displaystyle{\bigoplus_{k=1}^{n}\Bbbk x_k}$, the coefficient of $x_1^{p_1} \cdots x_n^{p_n}y_1^{q_1} \cdots y_n^{q_n}$ in $\widetilde{\phi}(x_1^{p_1} \cdots x_n^{p_n}y_1^{q_1} \cdots y_n^{q_n})$ is $b_{p_1, \cdots, p_n}b_{q_1, \cdots, q_n}$.
\end{enumerate}

Let $1 + \displaystyle{\sum_{i \geq 1}c_it^i}$ be the Taylor expansion of $\text{Tr}_{U(P)}(\widetilde{\phi},t)$, where $c_i \in \Bbbk$ for all $i \geq 1$. In accordance with the definition of the trace series and the above argument, the coefficient relating to the dimension of the degree $d$ component $c_d$ equals to the summation of all $b_{p_1, \cdots, p_n}b_{q_1, \cdots, q_n}$ ranging over $p_1 + \cdots + p_n + q_1 + \cdots + q_n = d$. Equivalently, $c_d = \displaystyle{\sum_{i+j=d}a_ia_j}$, and therefore, $(1 + \displaystyle{\sum_{i \geq 1}a_it^i})^2 = 1 + \displaystyle{\sum_{i \geq 1}c_it^i}$. 

Finally, since $P$ is a commutative polynomial ring, $\displaystyle{\text{Tr}_{P}(g,t) = \frac{1}{(1-\lambda_1t)^{c_1} \cdots (1-\lambda_mt)^{c_m}}}$, and consequently, 
\[
\text{Tr}_{U(P)}(\widetilde{\phi},t) = \frac{1}{(1-\lambda_1t)^{2c_1} \cdots (1-\lambda_mt)^{2c_m}},
\]
as desired. 
\end{proof}
With the assistance of Lemma 4.3, we are prepared to state a simple formula for the computation of the homological determinant of $\widetilde{\phi}$.

\begin{thm}
Let $P = \Bbbk[x_1, \cdots, x_n]$ be a quadratic Poisson algebra and let $U(P)$ be its Poisson enveloping algebra. Let $\phi$ be a finite-order graded Poisson automorphism of $P$ and let $\widetilde{\phi}$ be the induced graded automorphism of $U(P)$. Then
\[
\text{hdet} \widetilde{\phi} = (\det \phi\big\vert_{P_1})^2.
\]
\end{thm}
\begin{proof}
By \cite[Proposition 4.2]{JZ}, the graded automorphism $\widetilde{\phi}$ is rational over $\Bbbk$. Apply \cite[Lemma 2.6]{JZ}, the trace series
\[
\text{Tr}_{U(P)}(\widetilde{\phi},t) = (\text{hdet}\widetilde{\phi})^{-1}t^{-2n} + \textit{lower terms},
\]
when written as a Laurent series in $t^{-1}$. By Lemma 4.6, 
\begin{align*}
\text{Tr}_{U(P)}(\widetilde{\phi},t) =& \frac{1}{(1-\lambda_1t)^{2c_1} \cdots (1-\lambda_mt)^{2c_m}}\\ 
\smallskip
=& \frac{1}{\left(\det(\phi\big\vert_{P_1})\right)^2t^{2n} + \textit{lower terms}}\\
\smallskip
=& \left(\text{det}(\phi\big\vert_{P_1})^2\right)^{-1}t^{-2n} + \textit{lower terms}.
\end{align*}

Comparing the leading coefficient, $\text{hdet} \widetilde{\phi} = (\det \phi\big\vert_{P_1})^2$. 
\end{proof}

Combining \cite[Theorem 3.3]{JZ} and Theorem 4.4, we are able to provide an answer to Question (2) for quadratic Poisson structures on $\Bbbk[x_1, \cdots, x_n]$:

\begin{cor}
Let $P = \Bbbk[x_1, \cdots, x_n]$ be a quadratic Poisson algebra and let $U(P)$ be its Poisson enveloping algebra. Let $G$ be a finite subgroup of the graded Poisson automorphism group of $P$ and let $\widetilde{G} = \{\widetilde{\phi}: \phi \in G\}$ be the corresponding finite subgroup of the graded automorphism group of $U(P)$. If $G$ is generated by graded Poisson automorphisms $\phi_1, \cdots, \phi_m$ such that $\text{det}(\phi_i\big\vert_{P_1}) = \pm 1$, then $U(P)^{\widetilde{G}}$ is Artin-Schelter Gorenstein.
\end{cor}

\begin{proof}
This is an application of \cite[Theorem 3.3]{JZ} when we substitute the value of the homological determinant as per the formula provided in Theorem 4.4.
\end{proof}

\smallskip

We conclude this section with the following example.

\begin{example}
Let $P = \Bbbk[x_1,x_2]$ be the Poisson algebra $\{f, g\} = \displaystyle{\left(\frac{\partial f}{\partial x_1}\frac{\partial g}{\partial x_2} - \frac{\partial g}{\partial x_1}\frac{\partial f}{\partial x_2}\right)}qx_1x_2$, for some $q \neq 0$, for all $f, g \in P$. Let $\xi_n$ be a primitive $n$th root of unity. Consider the graded Poisson automorphism group $G = \langle \phi = [x_1 \mapsto \xi_nx_1, x_2 \mapsto x_2] \rangle$ of $P$, and its induced graded automorphism group $\widetilde{G} = \langle \widetilde{\phi} = [x_1 \mapsto \xi_nx_1, x_2 \mapsto x_2, y_1 \mapsto \xi_ny_1, y_2 \mapsto y_2] \rangle$ of $U(P)$. The invariant subalgebra $U(P)^{\widetilde{G}}$ is isomorphic to the graded $\Bbbk$-algebra generated by $a_1, \cdots, a_{n+1}$, $b$, $c$ subjecting to the following relations:
\begin{itemize}
    \item $[a_i,a_j]$,
    
    \item $a_ia_j = a_ka_{i+j-k}$,
    
    \item $a_{i+1}b-\displaystyle{\sum_{j=0}^{i}\binom{i}{j}q^jba_{i+1-j}}$,

    \item $a_{i+1}c-\displaystyle{\sum_{j=0}^{i}\binom{i}{j}q^j(c+npb)a_{i+1-j}}$.
\end{itemize}

According to Theorem 4.1 and Corollary 4.5, the invariant subalgebra $U(P)^{\widetilde{G}}$ is not Artin-Schelter regular, and is not Artin-Schelter Gorenstein except when $n = 2$.
\end{example}

\

\section{Future Work}
In this section, we remark on some additional findings encountered during the proof of the main results presented in this paper and put forth several avenues for future research.

\begin{quest}
Let $P = \Bbbk[x_1, \cdots, x_n]$ be a quadratic Poisson algebra and let $G \subseteq \text{PAut}_{\text{gr}}(P)$ be a finite subgroup. In the case of all quadratic Poisson structures that have been examined, Theorem 3.1, \cite[Theorem 4.4, 4.11, 4.17, 4.19]{GVW}, $P^G \cong P$ as Poisson algebras if and only if $G$ is trivial. Does this statement hold universally for all quadratic Poisson structures? Stated differently, is this the Shephard-Todd-Chevalley Theorem for quadratic Poisson algebras? A good starting place is the 13 non-unimodular quadratic Poisson structures classified in [\cite{DH}, Theorem 2].
\end{quest}
\begin{quest}
Let $P = \Bbbk[x_1, \cdots, x_n]$ be a Poisson algebra with its Poisson structure derived from the semiclassical limit of an Artin-Schelter regular algebra $A$. It appears that the invariant subalgebras of $P$ and the invariant subslagebras of $A$ bear a striking resemblance. One example is \cite[Theorem 4.5]{KKZ2} and \cite[Theorem 3.8]{GVW}: these two papers state an identical Shephard-Todd-Chevalley theorem for the skew polynomial rings and the Poisson algebras arising from them. Another example is \cite[Proposition 5.8]{KKZ2} and \cite[4.6-4.11]{GVW}: these two papers capture some significant similarities between the quantum matrix algebras and the Poisson algebras arising from them. Naturally, we inquire if there exists a comprehensive theory linking the invariants of the Poisson algebras and their corresponding Artin-Schelter regular algebras.
\end{quest}

\

\bibliographystyle{alpha}
\bibliography{reference}

\end{document}